\documentclass[pdflatex,reqno,12pt]{amsart}
\usepackage[a4paper,height=20cm,top=26mm,bottom=26mm,left=26mm,right=26mm]{geometry}

\usepackage{subfiles}
\usepackage{cleveref}

\usepackage{etex}
\usepackage{amsthm,amssymb}
\usepackage{mathrsfs}%
\usepackage{epic,eepic}
\usepackage{xypic}

\usepackage{here}
\usepackage{bm}
\usepackage[all]{xy}
\usepackage{tikz}
\usetikzlibrary{calc}

\numberwithin{equation}{section}

\newtheorem{thm}{Theorem}[section]
\newtheorem{cor}[thm]{Corollary}
\newtheorem{lem}[thm]{Lemma}
\newtheorem{prop}[thm]{Proposition}

\theoremstyle{definition}

\theoremstyle{remark}
\newtheorem{remark}[thm]{Remark}

\crefname{thm}{Theorem}{Theorems}
\crefname{cor}{Corollary}{Corollaries}
\crefname{lem}{Lemma}{Lemmas}
\crefname{prop}{Proposition}{Propositions}
\crefname{definition}{Definition}{Definitions}
\crefname{example}{Example}{Examples}
\crefname{claim}{Claim}{Claims}
\crefname{conjecture}{Conjecture}{Conjectures}
\crefname{remark}{Remark}{Remarks}
\crefname{figure}{Figure}{Figures}
\crefname{section}{Section}{Sections}
\crefname{subsection}{Section}{Sections}

\crefname{introthm}{Theorem}{Theorems}
\crefname{introcor}{Corollary}{Corollaries}
\crefname{introconj}{Conjecture}{Conjectures}

\def\C{{\mathbb C}}
\def\Z{{\mathbb Z}}
\def\ve{{\varepsilon}}

\newcommand\R{{\mathbb{R}}}
\newcommand\trop{{\mathrm{trop}}}

\newcommand\mg{{\mathfrak{g}}}
\newcommand\A{{\mathcal{A} }}

\newcommand\X{{\mathcal{X} }}

\newcommand\bY{{\boldsymbol{\mathcal{Y}} }}

\newfont{\bg}{cmr9 scaled\magstep4}
\newcommand{\bigzerol}{\smash{\lower1.0ex\hbox{\bg 0}}}

\newcommand\qline[2]{\draw[-,shorten >=2pt,shorten <=2pt] (#1) -- (#2);}
\newcommand\qarrow[2]{\draw[->,shorten >=2pt,shorten <=2pt] (#1) -- (#2) [thick];} 

\tikzset{
  mid arrow/.style={postaction={decorate,decoration={
        markings,
        mark=at position .5 with {\arrow[#1]{stealth}}
      }}},
}

\begin{document}
\title[Cluster realization of Weyl groups and $q$-characters]
{Cluster realization of Weyl groups and $q$-characters of quantum affine algebras}

\author[Rei Inoue]{Rei Inoue}
\address{Rei Inoue, Department of Mathematics and Informatics,
   Faculty of Science, Chiba University,
   Chiba 263-8522, Japan.}
\email{reiiy@math.s.chiba-u.ac.jp}

\date{March 8, 2020. Revised on December 18, 2020.}


\begin{abstract}
We consider an infinite quiver $Q(\mg)$ and a family of periodic quivers $Q_m(\mg)$ for a finite dimensional simple Lie algebra $\mg$ and $m \in \Z_{>1}$. 
The quiver $Q(\mg)$ is essentially same as what introduced in \cite{HL16} for the quantum affine algebra.
We construct the Weyl group $W(\mg)$ as a subgroup of the cluster modular group for $Q_m(\mg)$, in a similar way as \cite{IIO19},    
and study its applications to the $q$-characters of quantum non-twisted affine algebras $U_q(\hat{\mg})$ \cite{FR99}, and to the lattice $\mg$-Toda field theory \cite{IH00}. 
In particular, when $q$ is a root of unity, we prove that the $q$-character is invariant under the Weyl group action.
We also show that the $A$-variables for $Q(\mg)$ correspond to the $\tau$-function for the lattice $\mg$-Toda field equation.
\end{abstract}



\maketitle


\section{Introduction}

The {\em cluster algebra} was introduced around 2000 by Fomin and Zelevinsky \cite{FZ-CA1}. The heart of the algebra is algebraic operations called 
{\em mutations} of seeds, where a seed consists of a quiver (or a skew-symmetric matrix) and two tuples of variables ($A$-variables and $X$-variables). 
This algebra has been variously applied in mathematics and mathematical physics. One idea is to find an interesting sequence of mutations which preserves a given quiver $Q$, from which we define a rational map for the $A$- and $X$-variables. Such sequences of mutations constitute a group called the cluster modular group $\Gamma_Q$.
For example, the Laurent property of the Somos-4 sequence \cite{FZ-CA0}, the periodicity of the $T$-system and $Y$-system associated with the finite dimensional Lie algebra \cite{FZ03, K10, IIKNS, IIKKN} were proved by formulating the system with some quivers.

In \cite{IIO19}, for a symmetrizable Kac-Moody Lie algebra we constructed a family of periodic (weighted) quivers, and introduce the realization of the Weyl group as a subgroup of the cluster modular group. When the Lie algebra is classical finite type, this has application in higher Teichm\"uller space introduced by Fock and Goncharov \cite{FG03}. We showed that the $h$-periodic quiver is mutation equivalent to the quiver which gives the cluster structure for the higher Teichm\"uller space of once punctured disk with two marked points on the boundary, where $h$ is the Coxeter number of the Lie algebra. Based on a result by Goncharov and Shen in the case of $A_\ell$ \cite{GS16}, we further proved that the Weyl group action coincides with the geometric one defined in \cite{GS16} in the other classical cases.

In this paper, for a finite dimensional Lie algebra $\mg$ we introduce another family of quivers with the Weyl group realization, and study its applications.
Let $\ell$ be the rank of $\mg$, and set $S = \{1,2,\ldots,\ell \}$. We consider an infinite quiver $Q(\mg)$ and its periodic version $Q_m(\mg)$ for $m \in \Z_{>1}$. The infinite quiver $Q(\mg)$ is essentially same as what appeared in \cite{HL16}, and it was applied to study the cluster structures of representation for non-twisted quantum affine Lie algebra $U_q(\hat{\mg})$ \cite{HL16, HL16b} (see Remark \ref{rem:HL}).   
For the periodic quiver $Q_m(\mg)$, we construct sequences of mutations $R_i$ for $i \in S$ (Theorem \ref{R-action}), which generate the Weyl group $W(\mg)$ as a subgroup of the cluster modular group $\Gamma_{Q_m(\mg)}$ (Theorem \ref{thm:R-Weyl}). In fact, when $\mg$ is of simply-laced Dynkin types as $A_\ell$, $D_\ell$ and $E_{6,7,8}$, this Weyl group realization is same as what studied  in \cite{ILP16, IIO19}, whereas when $\mg$ is of non-simply-laced Dynkin type as $B_\ell$, $C_\ell$, $F_4$ and $G_2$, this gives a different cluster realization of $W(\mg)$. 
As same as in \cite{IIO19}, we show that each $R_i$ is a green sequence, and that the longest element in $W(\mg)$ gives the maximal green sequence (Theorem \ref{thm:R-Weyl}) in the sense of Keller \cite{Keller11}.

Further we present two applications of the quivers and the Weyl group realization:
the first one is to the $q$-characters for finite dimensional representations of quantum non-twisted affine algebras $U_q(\hat{\mg})$ introduced by Frenkel and Reshetikhin \cite{FR98,FR99}. The $q$-character $\chi_q$ is a ring homomorphism: 
$$
  \chi_q : \mathrm{Rep}\,U_q(\hat{\mg}) \to \mathbf{Y} := \Z[Y_{i,a_i}^{\pm 1}]_{i \in S, a_i \in \C^\times},
$$
and the image of $\chi_q$ is known to be
$$
\bigcap_{i \in S} \left(\Z[Y_{j,a}^{\pm 1}]_{j \neq i, a \in \C^\times} \otimes  \Z[Y_{i,b}(1+A_{i,bq_i}^{-1})]_{b \in \C^\times} \right),
$$
where the $A_{i,bq_i}$ are Laurent monomials in the $Y_{j,a}$. (See \S \ref{subsec:q-ch} for the precise definitions.)
For each $a \in \C^\times/q^{d\Z}$ 
we define a Poisson map from the variables $A_{i,aq_i}$ to the $X$-variables for $Q(\mg)$ (Proposition \ref{prop:PoissonaX-g}), where $d$ is defined at \eqref{eq:def-d}. The case that $q$ is a root of unity naturally corresponds to the quiver $Q_m(\mg)$. We extend the action of $W(\mg)$ on the $X$-variables to that on the variables $Y_{i,a q_i}$ (Theorem \ref{thm:WonY-g}), and prove that the $q$-character is invariant under the $W(\mg)$-action (Proposition \ref{prop:qch-weyl}).
When $q$ is generic, the $W(\mg)$-action is not well-defined, but we have an analog of $R_i$ which preserves the $q$-character (Proposition \ref{prop:qch-weyl-inf}).
The second application is to the lattice $\mg$-Toda field theory introduced in \cite{IH00}. We show that the $\tau$-function for the lattice $\mg$-Toda field equation is nothing but the $A$-variables for the infinite quiver $Q(\mg)$ (Proposition \ref{prop:Toda-A}).  
The two applications should be closely related, since the image of $\chi_q$ coincides with the intersection of the kernels of the screening operators for the $q$-deformed $\mathcal{W}$-algebras \cite{FR99} (also see the last part of \S \ref{subsec:q-ch}).  

In closing the introduction, we remark an interesting relation with the braid group action on the $\ell$-integral weight lattice introduced in \cite{CM05}. Roughly speaking, our Weyl group action on the {\it tropical} $X$-variables (Corollary \ref{cor:tropR}) corresponds to the braid group action on the $\ell$-root lattice. See Remark \ref{rem:braidacion}. It might be interesting to study this relation further.

\subsection*{Contents of the paper}

This paper is organized as follows. In \S 2, we briefly introduce the basic definitions in cluster algebras we use in this paper. We also fix the notations in Lie algebras and Weyl groups. In \S 3 we introduce the infinite quiver $Q(\mg)$ and the periodic one $Q_m(\mg)$, and study the realization of the Weyl group $W(\mg)$ as a subgroup of $\Gamma_{Q_m(\mg)}$. A remark on the case of non-twisted affine Lie algebra $\hat{\mg}$ is presented. In \S 4 and \S 5, we present the two applications of what constructed in \S 3.

\subsection*{Acknowledgement}

The author thanks Ryo Fujita, Tsukasa Ishibashi and Hironori Oya for valuable comments and discussions.  
The author is supported by JSPS KAKENHI Grant Number 16H03927, 18KK0071 and 19K03440. 
The author is grateful to anonymous referees for helpful comments and referring \cite{CM05}.

\section{Notation and Definitions}

\subsection{Seed mutation}\label{subsec:mutation}

Let $I$ be a finite set or a countably infinite set, and 
$\ve = (\ve_{ij})_{i,j \in I}$ be an integral skew-symmetric matrix.
When $I$ is infinite, we assume that for each $i \in I$ the number of $j \in I$ such that $\ve_{ij} \neq 0$ is finite. 
The matrix $\ve$ is called the \emph{exchange matrix}. 
The quiver $Q$ corresponding to $(I,\ve)$ is with the vertex set $I$, and it holds that $\ve_{ij} := \#\{\text{arrows from $i$ to $j$}\}  - \#\{\text{arrows from $j$ to $i$}\}$.

Let $\mathcal{F}$ be a field isomorphic to the field of rational functions over $\C$ in $|I|$ independent variables.
We recall the definition of the seed mutation by Fomin and Zelevinsky \cite{FZ-CA4} with the convention in \cite{FG06}. 
Let $\mathbf{X} = (X_i)_{i\in I}$ and $\mathbf{A} = (A_i)_{i\in I}$ be two tuples of algebraically independent elements in the field $\mathcal{F}$. The tuple $(Q,\mathbf{X},\mathbf{A})$ is called a \emph{seed}, and the pair $(Q,\mathbf{A})$ (resp. $(Q,\mathbf{X})$) is called an \emph{$A$-seed} (resp. an \emph{$X$-seed}). For $k \in I$, 
the mutation $\mu_k$ of the seed $\mu_k(\ve,\mathbf{X},\mathbf{A})
= (\ve', \mathbf{X}',\mathbf{A}')$ is given by 
\begin{align}\label{eq:e-mutation}
  &\ve_{ij}' = 
  \begin{cases}
    -\ve_{ij} & i=k \text{ or } j=k,
    \\
    \displaystyle{\ve_{ij} + \frac{|\ve_{ik}| \ve_{kj} + \ve_{ik} |\ve_{kj}|}{2}}
    & \text{otherwise},
  \end{cases}
  \\ \label{eq:X-mutation}
  &X_{i}' = 
  \begin{cases}
  X_k^{-1} & i=k,
  \\
  X_i (1 + X_k^{-\mathrm{sgn}(\ve_{ik})})^{-\ve_{ik}} & i \neq k,
  \end{cases}
  \\ \label{eq:A-mutation}
  &A_i' = 
  \begin{cases}
  \displaystyle{A_k^{-1} \left(\prod_{j:\ve_{kj}>0} A_j^{\ve_{kj}} 
       + \prod_{j:\ve_{kj}<0} A_j^{-\ve_{kj}} \right)} & i = k,
  \\ 
  A_i & i \neq k.
  \end{cases}
\end{align}  
It is easily checked that the obtained tuple $(\ve', \mathbf{X}',\mathbf{A}')$ is again a seed.
\\

\paragraph{\textbf{Poisson structure}}
For a seed $(\ve, \mathbf{A}, \mathbf{X})$,
let $\C(\mathbf{X})$ and $\C(\mathbf{A})$ be the rational functional fields generated by the $X$-variables and the $A$-variables respectively.
There is a log-canonical Poisson structure on $\C(\mathbf{X})$ and a closed two-form on $\C(\mathbf{A})$ compatible with seed mutation \cite{GSV03}, expressed as
\begin{align}
\label{Poisson-XA}
\{X_i,X_j\}=\ve_{ij} X_i X_j, \qquad \Omega_\A=\sum_{i,j \in I}\ve_{ij} \frac{\mathrm{d}A_i}{A_i}\wedge \frac{\mathrm{d}A_j}{A_j}. 
\end{align}
When the exchange matrix $\ve$ is non-degenerate, the form $\Omega_A$ is a symplectic form.
\\

\paragraph{\textbf{Cluster modular group}}

For a given quiver $Q$, a mutation sequence is a sequence of mutations associated with $Q$ and permutations of the vertex set $I$.
A mutation sequence is {\em trivial} if it preserves the seed.
The {\em cluster modular group} $\Gamma_Q$ is a group generated by mutation sequences which preserve $Q$, modulo trivial ones. 
\\

\paragraph{\textbf{Tropical $X$-variables}}

Let $\mathbb{P}=\mathbb{P}_\trop(u_1,u_2,\ldots,u_p) := \{\prod_{i=1}^{p} u_i^{a_i}; ~a_i \in \Z \}$ be the tropical semifield of rank $p$, equipped with the addition $\oplus$ and multiplication $\cdot$ as
$$
  \prod_{i=1}^{p} u_i^{a_i} \oplus \prod_{i=1}^{p} u_i^{b_i}
  = 
  \prod_{i=1}^{p} u_i^{\min(a_i,b_i)}, 
  \qquad
  \prod_{i=1}^{p} u_i^{a_i} \cdot \prod_{i=1}^{p} u_i^{b_i}
  =  
  \prod_{i=1}^{p} u_i^{a_i+b_i}.
$$
For an element $x = \prod_{i=1}^{p} u_i^{a_i} \in \mathbb{P}$, we define the tropical sign; $x$ is positive (resp. negative) if $a_i \geq 0$ (resp. $a_i \leq 0$) for all $i=1,\ldots,p$. We write $x>0$ (resp. $x <0$), when $x$ is positive (resp. negative). 

For a tropical $X$-seed $(Q, \mathbf{x}); ~\mathbf{x} \in \mathbb{P}^{|I|}$ and $k \in I$, the mutation $\mu_k (Q, \mathbf{x}) = (Q',\mathbf{x}')$ is given by \eqref{eq:e-mutation} and 
\begin{equation}\label{eq:trop-mutation}
x_{i}' = 
  \begin{cases}
  x_k^{-1} & i=k,
  \\
  x_i \cdot (1 \oplus x_k^{-\mathrm{sgn}(\ve_{ik})})^{-\ve_{ik}} & i \neq k.
  \end{cases}
\end{equation}

The followings are fundamental theorems which respectively state
the {\em sign coherence} of tropical $X$-variables and the {\em periodicity} of seeds. 
For a quiver $Q$ with a vertex set $I$, let $\mathbb{P}=\mathbb{P}_\trop(\mathbf{u})$ be a tropical semifield of rank $|I|$.
    
\begin{thm}[\cite{FZ-CA4,GHKK14}]
\label{thm:sign-coherence}
For any mutation sequence $\nu$ which mutates a tropical $X$-seed $(Q,\mathbf{u})$ into $(Q',\mathbf(x_i)_{i \in I})$, it holds that each $x_i \in \mathbb{P}$ is negative or positive. 
\end{thm}

\begin{thm}[\cite{IIKKN}]\label{thm:periodicity}
For a seed $(Q,\mathbf{X}, \mathbf{A})$, a tropical $X$-seed $(Q,\mathbf{u})$ and a mutation sequence $\nu$, the followings are equivalent:

  (i) $\nu (Q, \mathbf{X}, \mathbf{A})  = (Q, \mathbf{X}, \mathbf{A})$,
  \qquad
  (ii) $\nu (Q, \mathbf{u}) = (Q, \mathbf{u})$.
\end{thm}

For a quiver $Q$ of a vertex set $I$, we define a set $\mathcal{X}_Q(\mathbb{P}) := \{\mathbf{x} = (x_i)_{i \in I} \in \mathbb{P}^{|I|} \}$ where the cluster modular group $\Gamma_Q$ acts. We write $\mathcal{X}_Q^+(\mathbb{P})$ (resp. $\mathcal{X}_Q^-(\mathbb{P})$) for the subset of $\mathcal{X}_Q(\mathbb{P})$ consisting of positive (resp. negative) $\mathbf{x}$, i. e., the tropical sign of all $x_i$ of $\mathbf{x}$ is positive (resp. negative).
\\

\paragraph{\textbf{Green and Maximal green sequences}}

For a quiver $Q$, we again consider the tropical semifield $\mathbb{P}=\mathbb{P}_\trop(\mathbf{u})$ of rank $|I|$. 
We say that a sequence $\mathbf{i} = (i_1,i_2,\ldots,i_N)$ in $I$ is {\it green}, if
in the sequence of seeds 
$$
(Q[0],\mathbf{x}[0]) := (Q,\mathbf{u}) 
\stackrel{\mu_{i_1}}{\longmapsto} (Q[1],\mathbf{x}[1])
\stackrel{\mu_{i_2}}{\longmapsto} (Q[2],\mathbf{x}[2])
\stackrel{\mu_{i_3}}{\longmapsto} \cdots
\stackrel{\mu_{i_N}}{\longmapsto} (Q[N],\mathbf{x}[N]),
$$
it holds that $x[k]_{i_{k+1}} > 0$ for $k=0,1,\ldots,N-1$.
We say that the sequence $\mathbf{i}$ is {\it maximal green} 
if $\mathbf{i}$ is green and $\mathbf{x}[N] \in \mathcal{X}_Q^-(\mathbb{P})$.
These notions are essentially the same as the original ones in \cite{Keller11}.


\subsection{Lie algebras and Weyl groups}\label{subsec:Lie-alg}

Let $\mg$ be a finite-dimensional simple Lie algebra of rank $\ell$ and over $\C$. Define its rank set $S = \{1,2,\ldots,\ell \}$. Let $\mathfrak{h}$ be the Cartan subalgebra of $\mg$.
For $i \in S$, we write $\alpha_i$ for the $i$th simple root, and $\omega_i$ for the $i$th fundamental weight.
The Cartan matrix  
$\mathbf{C} = (C_{ij})_{i,j \in S}$
is given by 
$$
C_{ij} = 2 \, \frac{(\alpha_i \,,\,\alpha_j)}{(\alpha_i\,,\,\alpha_i)},
$$
where $( ~~ ,~~ )$ is an inner product. 
We  define
\begin{equation*}
  d_i = \frac{1}{2} ( \alpha_i \,,\, \alpha_i) ,
\end{equation*}
and a diagonal matrix $\mathbf{D}= \text{diag}(d_1, \cdots, d_\ell)$.
We write $d$ and $d'$ for the minimum and maximal value of $d_i$;
\begin{align}\label{eq:def-d}
d= \min\{ d_i ; ~i \in S\}, \quad d'= \max\{ d_i ; ~i \in S \}.
\end{align}
See Figure \ref{Dynkin-diagrams} for the definition of Dynkin diagrams in this paper.
We define a symmetrized Cartan matrix $\mathbf{B} = (B_{ij})_{1 \leq i,j \leq l}$ by
\begin{equation}
  \label{Cartan}
  \mathbf{B} 
  =
  \mathbf{D} \mathbf{C}.
  \qquad 
\end{equation}
We note that the matrix $\mathbf{D}$ becomes an identity matrix when $\mg$ 
has a simply-laced Dynkin diagram.

The Weyl group $W(\mg)$ associated with $\mg$ is group of the following presentation:
$$
W(\mg)=\langle r_i; ~i \in S \mid (r_i r_j)^{m_{ij}}=1 ~(i.j\in S) \rangle.
$$
Here $(m_{ij})_{i,j \in S}$ is a symmetric matrix satisfying $m_{ii}=1$ for all $i \in S$, and 
\[
\begin{tabular}{rccccc}
$C_{ij}C_{ji}:$ & $0$ & $1$ & $2$ & $3$ & $\geq 4$ \\
$m_{ij}:$         & $2$ & $3$ & $4$ & $6$ & $\infty$
\end{tabular}
\]
for $i \neq j$. Note that the case that $C_{ij} C_{ji} \geq 4$ does not appear when $\mg$ is finite. 
 
\begin{figure}[ht]
\begin{tikzpicture}
\begin{scope}[>=latex]
\draw (0,9) node{$A_\ell:$};
\draw (1,9) circle(2pt) coordinate(A1) node[above]{$1$};
\draw (2.5,9) circle(2pt) coordinate(A2) node[above]{$2$};
\draw (4,9) circle(2pt) coordinate(A3) node[above]{$3$};
\draw (8,9) circle(2pt) coordinate(A4) node[above]{$\ell-1$};
\draw (9.5,9) circle(2pt) coordinate(A5) node[above]{$\ell$};
\qline{A2}{A1}
\qline{A3}{A2}
\draw[-,shorten >=2pt,shorten <=2pt] (5.5,9) -- (A3) [thick];
\draw (6,9) node{$\dots$};
\draw[-,shorten >=2pt,shorten <=2pt] (A4) -- (6.5,9) [thick];
\qline{A5}{A4}

\draw (0,7.5) node{$B_\ell:$};
\draw (1,7.5) circle(2pt) coordinate(B1) node[above]{$1$};
\draw (2.5,7.5) circle(2pt) coordinate(B2) node[above]{$2$};
\draw (4,7.5) circle(2pt) coordinate(B3) node[above]{$3$};
\draw (8,7.5) circle(2pt) coordinate(B4) node[above]{$\ell-1$};
\draw (9.5,7.5) circle(2pt) coordinate(B5) node[above]{$\ell$};
\qline{B2}{B1}
\qline{B3}{B2}
\draw[-,shorten >=2pt,shorten <=2pt] (5.5,7.5) -- (B3);
\draw (6,7.5) node{$\dots$};
\draw[-,shorten >=2pt,shorten <=2pt] (B4) -- (6.5,7.5);
\draw[-,shorten >=2pt,shorten <=2pt] (8,7.475) -- (9.5,7.475);
\draw[-,shorten >=2pt,shorten <=2pt] (8,7.525) -- (9.5,7.525);
\draw (8.75,7.5) node{$>$};

\draw (0,6) node{$C_\ell:$};
\draw (1,6) circle(2pt) coordinate(C1) node[above]{$1$};
\draw (2.5,6) circle(2pt) coordinate(C2) node[above]{$2$};
\draw (4,6) circle(2pt) coordinate(C3) node[above]{$3$};
\draw (8,6) circle(2pt) coordinate(C4) node[above]{$\ell-1$};
\draw (9.5,6) circle(2pt) coordinate(C5) node[above]{$\ell$};
\qline{C2}{C1}
\qline{C3}{C2}
\draw[-,shorten >=2pt,shorten <=2pt] (5.5,6) -- (C3);
\draw (6,6) node{$\dots$};
\draw[-,shorten >=2pt,shorten <=2pt] (C4) -- (6.5,6);
\draw[-,shorten >=2pt,shorten <=2pt] (8,5.975) -- (9.5,5.975);
\draw[-,shorten >=2pt,shorten <=2pt] (8,6.025) -- (9.5,6.025);
\draw (8.75,6) node{$<$};

\draw (-0.2,4.5) node{$D_{\ell \geq 4}:$};
\draw (1,4.5) circle(2pt) coordinate(D1) node[above]{$1$};
\draw (2.5,4.5) circle(2pt) coordinate(D2) node[above]{$2$};
\draw (4,4.5) circle(2pt) coordinate(D3) node[above]{$3$};
\draw (8,4.5) circle(2pt) coordinate(D4) node[above]{$\ell-2$};
\draw (9.5,4.5) circle(2pt) coordinate(D5) node[above]{$\ell-1$};
\draw (8,3) circle(2pt) coordinate(D6) node[right]{$\ell$};
\qline{D2}{D1}
\qline{D3}{D2}
\draw[-,shorten >=2pt,shorten <=2pt] (5.5,4.5) -- (D3);
\draw (6,4.5) node{$\dots$};
\draw[-,shorten >=2pt,shorten <=2pt] (D4) -- (6.5,4.5);
\qline{D5}{D4}
\qline{D6}{D4}

\draw (-0.4,2) node{$E_{\ell=6,7,8}:$}; 
\draw (1,2) circle(2pt) coordinate(E1) node[above]{$1$};
\draw (2.5,2) circle(2pt) coordinate(E2) node[above]{$2$};
\draw (4,2) circle(2pt) coordinate(E3) node[above]{$3$};
\draw (5.5,2) circle(2pt) coordinate(E4) node[above]{$5$};
\draw (8,2) circle(2pt) coordinate(E5) node[above]{$\ell-1$};
\draw (9.5,2) circle(2pt) coordinate(E6) node[above]{$\ell$};
\draw (4,0.5) circle(2pt) coordinate(E7) node[right]{$4$};
\qline{E2}{E1}
\qline{E3}{E2}
\qline{E4}{E3}
\draw (6,2) node{$\dots$};
\draw[-,shorten >=2pt,shorten <=2pt] (E5) -- (6.5,2);
\qline{E6}{E5}
\qline{E7}{E3}

\draw (0,-0.5) node{$F_4:$};
\draw (1,-0.5) circle(2pt) coordinate(F1) node[above]{$1$};
\draw (2.5,-0.5) circle(2pt) coordinate(F2) node[above]{$2$};
\draw (4,-0.5) circle(2pt) coordinate(F3) node[above]{$3$};
\draw (5.5,-0.5) circle(2pt) coordinate(F4) node[above]{$4$};
\qline{F2}{F1}
\qline{F4}{F3}
\draw[-,shorten >=2pt,shorten <=2pt] (2.5,-0.475) -- (4,-0.475);
\draw[-,shorten >=2pt,shorten <=2pt] (2.5,-0.525) -- (4,-0.525);
\draw (3.25,-0.5) node{$>$};

\draw (7,-0.5) node{$G_2:$};
\draw (8,-0.5) circle(2pt) coordinate(G1) node[above]{$1$};
\draw (9.5,-0.5) circle(2pt) coordinate(G2) node[above]{$2$};
\qline{G2}{G1}
\draw[-,shorten >=2pt,shorten <=2pt] (8,-0.45) -- (9.5,-0.45); 
\draw[-,shorten >=2pt,shorten <=2pt] (8,-0.55) -- (9.5,-0.55); 
\draw (8.75,-0.5) node{$<$};

\end{scope}
\end{tikzpicture}
\caption{Dynkin diagrams for $\mg$}
\label{Dynkin-diagrams}
\end{figure}

\section{Cluster realization of Weyl groups}

\subsection{The infinite quiver $Q(\mg)$}

For each $\mg$ we define an infinite quiver $Q(\mg)$ of a vertex set $I = \{v^i_n; ~i \in S, n \in d \Z \}$.
\\

\paragraph{\textbf{The case of $A_\ell$, $D_\ell$, $E_{6,7,8}$}}
For $\mg$ of simply-laced Dynkin diagram, i.e., $\mathbf{D} = \mathbb{I}_\ell$, we set a Dynkin quiver $C(\mg)$ as Figure \ref{Dynkin-quivers}, and write $\ve = (\ve_{ij})_{i,j \in S}$ for the corresponding exchange matrix.
We define an infinite quiver $Q(\mg)$ of a vertex set $I = \{v_n^i;~ i \in S, n \in \Z \}$ in the following way:
\begin{itemize}
\item 
we have an arrow $v_n^i \to v_{n+1}^{i}$,
\item
if $\ve_{ij} = 1$, we have an arrow $v_n^i \to v_n^{j}$ and an arrow $v_{n+1}^j \to v_n^{i}$.
\end{itemize}
See Figure \ref{TQAn} for the quiver $Q(A_\ell)$.
\\

\begin{figure}[ht]
\begin{tikzpicture}
\begin{scope}[>=latex]
\draw (0,6) node{$A_\ell:$};
\draw (1,6) circle(2pt) coordinate(A1) node[above]{$1$};
\draw (2.5,6) circle(2pt) coordinate(A2) node[above]{$2$};
\draw (4,6) circle(2pt) coordinate(A3) node[above]{$3$};
\draw (8,6) circle(2pt) coordinate(A4) node[above]{$\ell-1$};
\draw (9.5,6) circle(2pt) coordinate(A5) node[above]{$\ell$};
\qarrow{A2}{A1}
\qarrow{A3}{A2}
\draw[->,shorten >=2pt,shorten <=2pt] (5.5,6) -- (A3) [thick];
\draw (6,6) node{$\dots$};
\draw[->,shorten >=2pt,shorten <=2pt] (A4) -- (6.5,6) [thick];
\qarrow{A5}{A4}

\draw (-0.2,4.5) node{$D_{s \geq 4}:$};
\draw (1,4.5) circle(2pt) coordinate(D1) node[above]{$1$};
\draw (2.5,4.5) circle(2pt) coordinate(D2) node[above]{$2$};
\draw (4,4.5) circle(2pt) coordinate(D3) node[above]{$3$};
\draw (8,4.5) circle(2pt) coordinate(D4) node[above]{$\ell-2$};
\draw (9.5,4.5) circle(2pt) coordinate(D5) node[above]{$\ell-1$};
\draw (8,3) circle(2pt) coordinate(D6) node[right]{$\ell$};
\qarrow{D2}{D1}
\qarrow{D3}{D2}
\draw[->,shorten >=2pt,shorten <=2pt] (5.5,4.5) -- (D3) [thick];
\draw (6,4.5) node{$\dots$};
\draw[->,shorten >=2pt,shorten <=2pt] (D4) -- (6.5,4.5) [thick];
\qarrow{D5}{D4}
\qarrow{D6}{D4}

\draw (-0.4,2) node{$E_{s=6,7,8}:$}; 
\draw (1,2) circle(2pt) coordinate(E1) node[above]{$1$};
\draw (2.5,2) circle(2pt) coordinate(E2) node[above]{$2$};
\draw (4,2) circle(2pt) coordinate(E3) node[above]{$3$};
\draw (5.5,2) circle(2pt) coordinate(E4) node[above]{$5$};
\draw (8,2) circle(2pt) coordinate(E5) node[above]{$\ell-1$};
\draw (9.5,2) circle(2pt) coordinate(E6) node[above]{$\ell$};
\draw (4,0.5) circle(2pt) coordinate(E7) node[right]{$4$};
\qarrow{E2}{E1}
\qarrow{E3}{E2}
\qarrow{E4}{E3}
\draw (6,2) node{$\dots$};
\draw[->,shorten >=2pt,shorten <=2pt] (E5) -- (6.5,2) [thick];
\qarrow{E6}{E5}
\qarrow{E7}{E3}

\end{scope}
\end{tikzpicture}
\caption{Dynkin quivers $C(\mathfrak{g})$ for $\mathfrak{g} = A_\ell, D_\ell$ and $E_{6,7,8}$}
\label{Dynkin-quivers}
\end{figure}
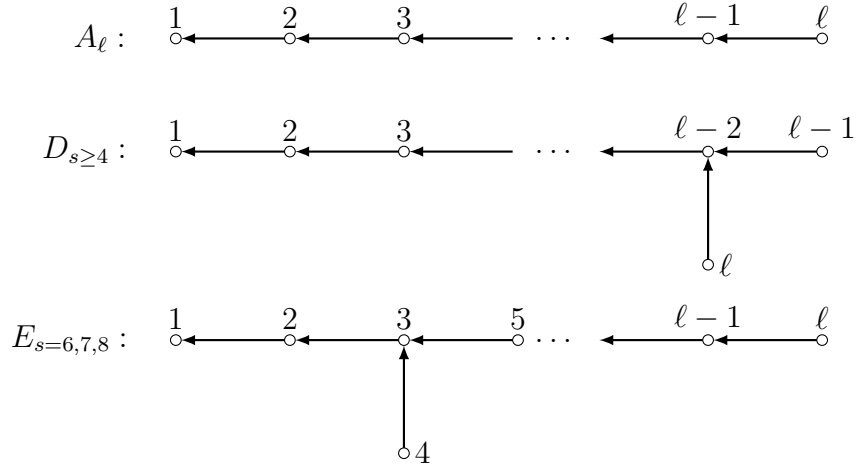

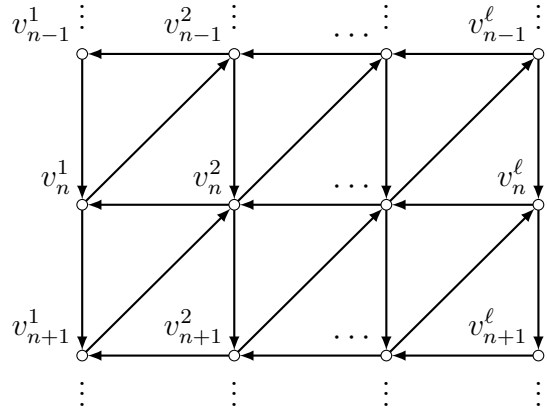
\begin{figure}[ht]
\begin{tikzpicture}
\begin{scope}[>=latex]
\draw (0,8.6) node{$\vdots$};
\draw (2,8.6) node{$\vdots$};
\draw (4,8.6) node{$\vdots$};
\draw (6,8.6) node{$\vdots$};

\draw (0,8) circle(2pt) coordinate(A1) node[above left]{$v^1_{n-1}$};
\draw (2,8) circle(2pt) coordinate(A2) node[above left]{$v^2_{n-1}$};
\draw (4,8) circle(2pt) coordinate(A3) node[above left]{$\cdots$};
\draw (6,8) circle(2pt) coordinate(A4) node[above left]{$v^\ell_{n-1}$};
\qarrow{A2}{A1}
\qarrow{A3}{A2}
\qarrow{A4}{A3}
\draw (0,6) circle(2pt) coordinate(B1) node[above left]{$v^1_n$};
\draw (2,6) circle(2pt) coordinate(B2) node[above left]{$v^2_n$};
\draw (4,6) circle(2pt) coordinate(B3) node[above left]{$\cdots$};
\draw (6,6) circle(2pt) coordinate(B4) node[above left]{$v^\ell_n$};
\qarrow{B2}{B1}
\qarrow{B3}{B2}
\qarrow{B4}{B3}
\qarrow{A1}{B1}
\qarrow{A2}{B2}
\qarrow{A3}{B3}
\qarrow{A4}{B4}
\qarrow{B1}{A2}
\qarrow{B2}{A3}
\qarrow{B3}{A4}
\draw (0,4) circle(2pt) coordinate(C1) node[above left]{$v^1_{n+1}$};
\draw (2,4) circle(2pt) coordinate(C2) node[above left]{$v^2_{n+1}$};
\draw (4,4) circle(2pt) coordinate(C3) node[above left]{$\cdots$};
\draw (6,4) circle(2pt) coordinate(C4) node[above left]{$v^\ell_{n+1}$};
\qarrow{C2}{C1}
\qarrow{C3}{C2}
\qarrow{C4}{C3}
\qarrow{B1}{C1}
\qarrow{B2}{C2}
\qarrow{B3}{C3}
\qarrow{B4}{C4}
\qarrow{C1}{B2}
\qarrow{C2}{B3}
\qarrow{C3}{B4}

\draw (0,3.6) node{$\vdots$};
\draw (2,3.6) node{$\vdots$};
\draw (4,3.6) node{$\vdots$};
\draw (6,3.6) node{$\vdots$};

\end{scope}
\end{tikzpicture}
\caption{The quiver $Q(A_\ell)$.}
\label{TQAn}
\end{figure}

\paragraph{\textbf{The case of $B_\ell$}}
We have $\mathbf{D}=(\underbrace{1,\ldots,1}_{\ell-1},\frac{1}{2})$, thus $d = \frac{1}{2}$. 
We define $Q(B_\ell)$ of a vertex set $I = \{v_n^i;~ i \in S, n \in \Z/2 \}$ in the following way: 
\begin{itemize}
\item 
prepare two copies of $Q(A_{\ell-1})$, where the vertex set of the first and second copies are respectively $I_1 = \{v_n^i; i \in S \setminus \{\ell\}, n \in \Z \}$ and $I_2 = \{v_n^i; i \in S \setminus \{\ell \}, n \in \Z +\frac{1}{2}\}$. 
Also prepare vertices $v_n^\ell$ for $n \in \Z/2$.
\item
we have an arrow $v_n^\ell \to v_{n+\frac{1}{2}}^{\ell}$, for $n \in \Z/2$, 
\item
we have an arrow $v_n^\ell \to v_{n-\frac{1}{2}}^{\ell-1}$ and an arrow
$v_{n+\frac{1}{2}}^{\ell-1} \to v_n^{\ell}$, for $n \in \Z/2$. 
\end{itemize}
See Figure \ref{TQB4} for the quiver $Q(B_4)$.
\\

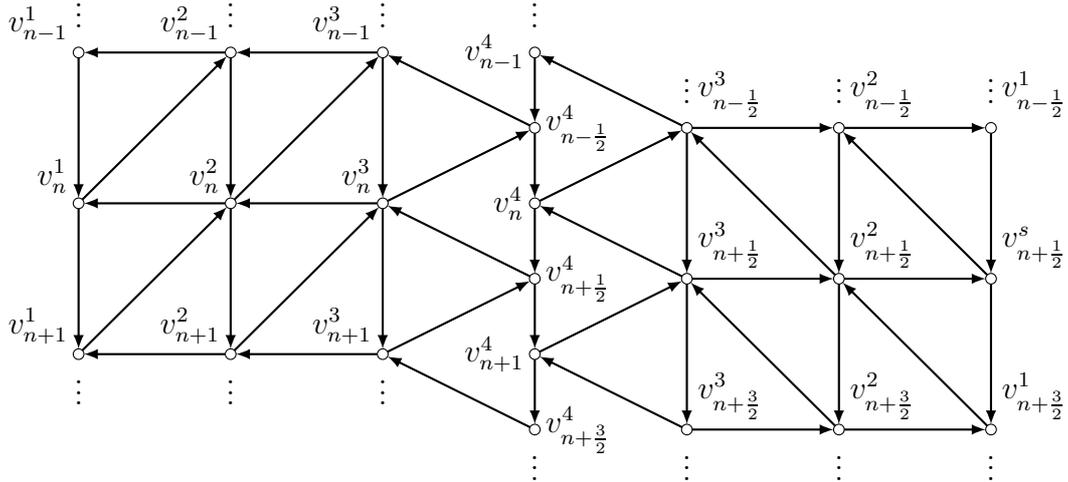
\begin{figure}[ht]
\begin{tikzpicture}
\begin{scope}[>=latex]
\draw (0,8.6) node{$\vdots$};
\draw (2,8.6) node{$\vdots$};
\draw (4,8.6) node{$\vdots$};
\draw (6,8.6) node{$\vdots$};

\draw (0,8) circle(2pt) coordinate(A1) node[above left]{$v^1_{n-1}$};
\draw (2,8) circle(2pt) coordinate(A2) node[above left]{$v^2_{n-1}$};
\draw (4,8) circle(2pt) coordinate(A3) node[above left]{$v^3_{n-1}$};
\draw (6,8) circle(2pt) coordinate(A4) node[left]{$v^4_{n-1}$};
\qarrow{A2}{A1}
\qarrow{A3}{A2}
\draw (0,6) circle(2pt) coordinate(B1) node[above left]{$v^1_n$};
\draw (2,6) circle(2pt) coordinate(B2) node[above left]{$v^2_n$};
\draw (4,6) circle(2pt) coordinate(B3) node[above left]{$v^3_n$};
\draw (6,6) circle(2pt) coordinate(B4) node[left]{$v^4_n$};
\qarrow{B2}{B1}
\qarrow{B3}{B2}
%
\qarrow{A1}{B1}
\qarrow{A2}{B2}
\qarrow{A3}{B3}
%
\qarrow{B1}{A2}
\qarrow{B2}{A3}
\draw (0,4) circle(2pt) coordinate(C1) node[above left]{$v^1_{n+1}$};
\draw (2,4) circle(2pt) coordinate(C2) node[above left]{$v^2_{n+1}$};
\draw (4,4) circle(2pt) coordinate(C3) node[above left]{$v^3_{n+1}$};
\draw (6,4) circle(2pt) coordinate(C4) node[left]{$v^4_{n+1}$};
\qarrow{C2}{C1}
\qarrow{C3}{C2}
%
\qarrow{B1}{C1}
\qarrow{B2}{C2}
\qarrow{B3}{C3}
%
\qarrow{C1}{B2}
\qarrow{C2}{B3}

\draw (0,3.6) node{$\vdots$};
\draw (2,3.6) node{$\vdots$};
\draw (4,3.6) node{$\vdots$};


\draw (8,7.6) node{$\vdots$};
\draw (10,7.6) node{$\vdots$};
\draw (12,7.6) node{$\vdots$};

\draw (6,7) circle(2pt) coordinate(A11) node[right]{$v^4_{n-\frac{1}{2}}$};
\draw (8,7) circle(2pt) coordinate(A21) node[above right]{$v^3_{n-\frac{1}{2}}$};
\draw (10,7) circle(2pt) coordinate(A31) node[above right]{$v^2_{n-\frac{1}{2}}$};
\draw (12,7) circle(2pt) coordinate(A41) node[above right]{$v^1_{n-\frac{1}{2}}$};
\qarrow{A21}{A31}
\qarrow{A31}{A41}
\draw (6,5) circle(2pt) coordinate(B11) node[right]{$v^4_{n+\frac{1}{2}}$};
\draw (8,5) circle(2pt) coordinate(B21) node[above right]{$v^3_{n+\frac{1}{2}}$};
\draw (10,5) circle(2pt) coordinate(B31) node[above right]{$v^2_{n+\frac{1}{2}}$};
\draw (12,5) circle(2pt) coordinate(B41) node[above right]{$v^s_{n+\frac{1}{2}}$};
\qarrow{B21}{B31}
\qarrow{B31}{B41}
\qarrow{A21}{B21}
\qarrow{A31}{B31}
\qarrow{A41}{B41}
\qarrow{A4}{A11}
\qarrow{A11}{B4}
\qarrow{B4}{B11}
\qarrow{B11}{C4}
%
\qarrow{B31}{A21}
\qarrow{B41}{A31}
\draw (6,3) circle(2pt) coordinate(C11) node[right]{$v^4_{n+\frac{3}{2}}$};
\draw (8,3) circle(2pt) coordinate(C21) node[above right]{$v^3_{n+\frac{3}{2}}$};
\draw (10,3) circle(2pt) coordinate(C31) node[above right]{$v^2_{n+\frac{3}{2}}$};
\draw (12,3) circle(2pt) coordinate(C41) node[above right]{$v^1_{n+\frac{3}{2}}$};
\qarrow{C21}{C31}
\qarrow{C31}{C41}
\qarrow{C4}{C11}

\qarrow{B21}{C21}
\qarrow{B31}{C31}
\qarrow{B41}{C41}
%
\qarrow{C31}{B21}
\qarrow{C41}{B31}

\qarrow{C11}{C3}
\qarrow{C3}{B11}
\qarrow{B11}{B3}
\qarrow{B3}{A11}
\qarrow{A11}{A3}

\qarrow{C21}{C4}
\qarrow{C4}{B21}
\qarrow{B21}{B4}
\qarrow{B4}{A21}
\qarrow{A21}{A4}

\draw (6,2.6) node{$\vdots$};
\draw (8,2.6) node{$\vdots$};
\draw (10,2.6) node{$\vdots$};
\draw (12,2.6) node{$\vdots$};

\end{scope}
\end{tikzpicture}
\caption{The quiver $Q(B_4)$.}
\label{TQB4}
\end{figure}

\paragraph{\textbf{The case of $C_\ell$}}
We have $\mathbf{D}=(\underbrace{1,\ldots,1}_{\ell-1},2)$, thus $d = 1$. 
We define $Q(C_\ell)$ of a vertex set $I = \{v_n^i;~ i \in S, n \in \Z \}$ in the following way: 
\begin{itemize}
\item 
prepare the quiver $Q(A_{\ell-1})$ and vertices $v_n^\ell$ for $n \in \Z$,
\item
we have an arrow $v_n^\ell \to v_{n+2}^{\ell}$, for $n \in \Z$, 
\item
we have an arrow $v_n^\ell \to v_n^{\ell-1}$ and an arrow
$v_{n+2}^{\ell-1} \to v_n^{\ell}$, for $n \in \Z$. 
\end{itemize}
See Figure \ref{TQC4} for the quiver $Q(C_4)$.
\\

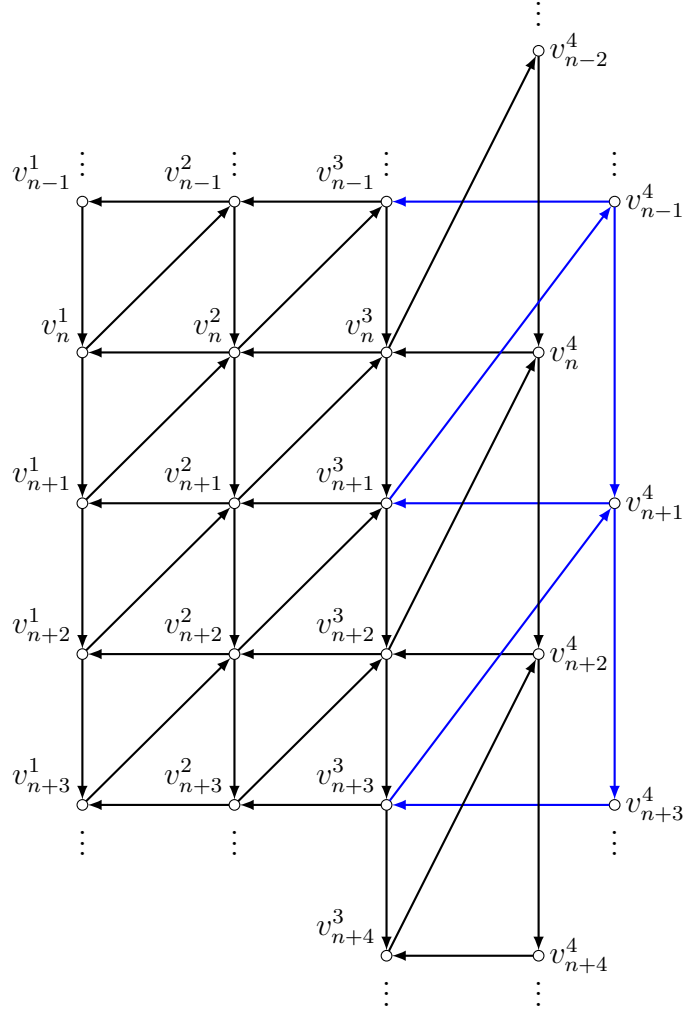
\begin{figure}[ht]
\begin{tikzpicture}
\begin{scope}[>=latex]
\draw (0,8.6) node{$\vdots$};
\draw (2,8.6) node{$\vdots$};
\draw (4,8.6) node{$\vdots$};
\draw (6,10.6) node{$\vdots$};
\draw (7,8.6) node{$\vdots$};

\draw (6,10) circle(2pt) coordinate(Z4) node[right]{$v^4_{n-2}$};

\draw (0,8) circle(2pt) coordinate(A1) node[above left]{$v^1_{n-1}$};
\draw (2,8) circle(2pt) coordinate(A2) node[above left]{$v^2_{n-1}$};
\draw (4,8) circle(2pt) coordinate(A3) node[above left]{$v^3_{n-1}$};
\draw (7,8) circle(2pt) coordinate(A4) node[right]{$v^4_{n-1}$};
\qarrow{A2}{A1}
\qarrow{A3}{A2}
{\color{blue}
\qarrow{A4}{A3}
}
\draw (0,6) circle(2pt) coordinate(B1) node[above left]{$v^1_n$};
\draw (2,6) circle(2pt) coordinate(B2) node[above left]{$v^2_n$};
\draw (4,6) circle(2pt) coordinate(B3) node[above left]{$v^3_n$};
\draw (6,6) circle(2pt) coordinate(B4) node[right]{$v^4_n$};
\qarrow{B2}{B1}
\qarrow{B3}{B2}
\qarrow{B4}{B3}
\qarrow{A1}{B1}
\qarrow{A2}{B2}
\qarrow{A3}{B3}
\qarrow{Z4}{B4}
\qarrow{B1}{A2}
\qarrow{B2}{A3}
\qarrow{B3}{Z4}
\draw (0,4) circle(2pt) coordinate(C1) node[above left]{$v^1_{n+1}$};
\draw (2,4) circle(2pt) coordinate(C2) node[above left]{$v^2_{n+1}$};
\draw (4,4) circle(2pt) coordinate(C3) node[above left]{$v^3_{n+1}$};
\draw (7,4) circle(2pt) coordinate(C4) node[right]{$v^4_{n+1}$};
\qarrow{C2}{C1}
\qarrow{C3}{C2}
{\color{blue}
\qarrow{C4}{C3}
}
\qarrow{B1}{C1}
\qarrow{B2}{C2}
\qarrow{B3}{C3}
%
\qarrow{C1}{B2}
\qarrow{C2}{B3}
{\color{blue}
\qarrow{C3}{A4}
}

\draw (0,2) circle(2pt) coordinate(D1) node[above left]{$v^1_{n+2}$};
\draw (2,2) circle(2pt) coordinate(D2) node[above left]{$v^2_{n+2}$};
\draw (4,2) circle(2pt) coordinate(D3) node[above left]{$v^3_{n+2}$};
\draw (6,2) circle(2pt) coordinate(D4) node[right]{$v^4_{n+2}$};
\qarrow{D2}{D1}
\qarrow{D3}{D2}
\qarrow{D4}{D3}

\qarrow{C1}{D1}
\qarrow{C2}{D2}
\qarrow{C3}{D3}
%
\qarrow{D1}{C2}
\qarrow{D2}{C3}
\qarrow{D3}{B4}
\draw (0,0) circle(2pt) coordinate(E1) node[above left]{$v^1_{n+3}$};
\draw (2,0) circle(2pt) coordinate(E2) node[above left]{$v^2_{n+3}$};
\draw (4,0) circle(2pt) coordinate(E3) node[above left]{$v^3_{n+3}$};
\draw (7,0) circle(2pt) coordinate(E4) node[right]{$v^4_{n+3}$};
\qarrow{E2}{E1}
\qarrow{E3}{E2}
{\color{blue}
\qarrow{E4}{E3}
}
\qarrow{D1}{E1}
\qarrow{D2}{E2}
\qarrow{D3}{E3}
\qarrow{E1}{D2}
\qarrow{E2}{D3}
{\color{blue}
\qarrow{E3}{C4}
}

\draw (6,-2) circle(2pt) coordinate(Y4) node[right]{$v^4_{n+4}$};
\draw (4,-2) circle(2pt) coordinate(Y3) node[above left]{$v^3_{n+4}$};

\qarrow{E3}{Y3}
\qarrow{Y4}{Y3}
\qarrow{B4}{D4}
\qarrow{D4}{Y4}
\qarrow{Y3}{D4}

{\color{blue}
\qarrow{A4}{C4}
\qarrow{C4}{E4}
}

\draw (0,-0.4) node{$\vdots$};
\draw (2,-0.4) node{$\vdots$};
\draw (4,-2.4) node{$\vdots$};
\draw (6,-2.4) node{$\vdots$};
\draw (7,-0.4) node{$\vdots$};

\end{scope}
\end{tikzpicture}
\caption{The quiver $Q(C_4)$. Arrows among the $v^3_n$ and $v^4_n$ are color-coded.}
\label{TQC4}
\end{figure}

\paragraph{\textbf{The case of $F_4$}}
We have $\mathbf{D}=(1,1,\frac{1}{2},\frac{1}{2})$, thus $d = \frac{1}{2}$. 
We define $Q(F_4)$ of a vertex set $I = \{v_n^i;~ i \in S, n \in \Z/2 \}$ in the following way: 
\begin{itemize}
\item 
prepare a copy of $Q(B_{3})$, and vertices $v_n^4$ for $n \in \Z/2$.  
\item
we have an arrow $v_n^4 \to v_{n+\frac{1}{2}}^{4}$, for $n \in \Z/2$, 
\item
we have an arrow $v_n^4 \to v_n^3$ and an arrow
$v_{n+\frac{1}{2}}^{3} \to v_n^{4}$, for $n \in \Z/2$. 
\end{itemize}
See Figure \ref{TQF4} for the quiver $Q(F_4)$.
\\

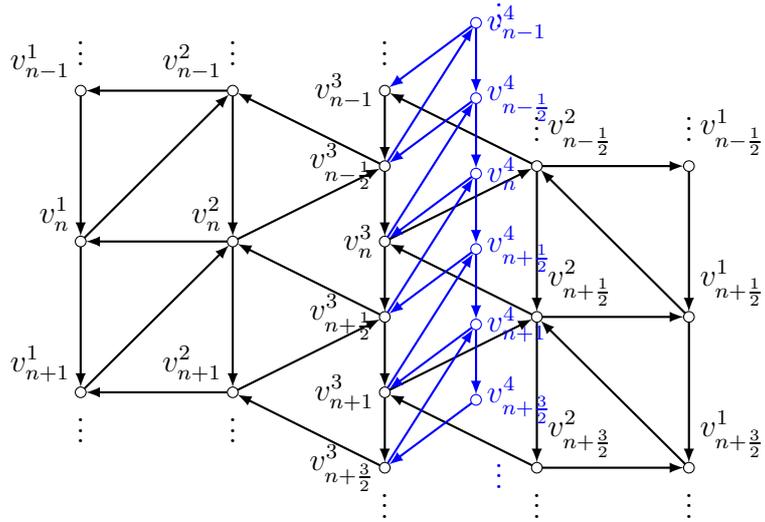
\begin{figure}[ht]
\begin{tikzpicture}
\begin{scope}[>=latex]
\draw (2,8.6) node{$\vdots$};
\draw (4,8.6) node{$\vdots$};
\draw (6,8.6) node{$\vdots$};

\draw (2,8) circle(2pt) coordinate(A2) node[above left]{$v^1_{n-1}$};
\draw (4,8) circle(2pt) coordinate(A3) node[above left]{$v^2_{n-1}$};
\draw (6,8) circle(2pt) coordinate(A4) node[left]{$v^3_{n-1}$};
\qarrow{A3}{A2}
\draw (2,6) circle(2pt) coordinate(B2) node[above left]{$v^1_n$};
\draw (4,6) circle(2pt) coordinate(B3) node[above left]{$v^2_n$};
\draw (6,6) circle(2pt) coordinate(B4) node[left]{$v^3_n$};
\qarrow{B3}{B2}
%
\qarrow{A2}{B2}
\qarrow{A3}{B3}
%
\qarrow{B2}{A3}
%
\draw (2,4) circle(2pt) coordinate(C2) node[above left]{$v^1_{n+1}$};
\draw (4,4) circle(2pt) coordinate(C3) node[above left]{$v^2_{n+1}$};
\draw (6,4) circle(2pt) coordinate(C4) node[left]{$v^3_{n+1}$};
\qarrow{C3}{C2}
%
\qarrow{B2}{C2}
\qarrow{B3}{C3}
%
\qarrow{C2}{B3}

\draw (2,3.6) node{$\vdots$};
\draw (4,3.6) node{$\vdots$};


\draw (8,7.6) node{$\vdots$};
\draw (10,7.6) node{$\vdots$};

\draw (6,7) circle(2pt) coordinate(A11) node[left]{$v^3_{n-\frac{1}{2}}$};
\draw (8,7) circle(2pt) coordinate(A21) node[above right]{$v^2_{n-\frac{1}{2}}$};
\draw (10,7) circle(2pt) coordinate(A31) node[above right]{$v^1_{n-\frac{1}{2}}$};
\qarrow{A21}{A31}
\draw (6,5) circle(2pt) coordinate(B11) node[left]{$v^3_{n+\frac{1}{2}}$};
\draw (8,5) circle(2pt) coordinate(B21) node[above right]{$v^2_{n+\frac{1}{2}}$};
\draw (10,5) circle(2pt) coordinate(B31) node[above right]{$v^1_{n+\frac{1}{2}}$};
\qarrow{B21}{B31}
%
\qarrow{A21}{B21}
\qarrow{A31}{B31}
%
\qarrow{A4}{A11}
\qarrow{A11}{B4}
\qarrow{B4}{B11}
\qarrow{B11}{C4}
%
\qarrow{B31}{A21}
%
\draw (6,3) circle(2pt) coordinate(C11) node[left]{$v^3_{n+\frac{3}{2}}$};
\draw (8,3) circle(2pt) coordinate(C21) node[above right]{$v^2_{n+\frac{3}{2}}$};
\draw (10,3) circle(2pt) coordinate(C31) node[above right]{$v^1_{n+\frac{3}{2}}$};
\qarrow{C21}{C31}
%
\qarrow{C4}{C11}

\qarrow{B21}{C21}
\qarrow{B31}{C31}
%
\qarrow{C31}{B21}

\qarrow{C11}{C3}
\qarrow{C3}{B11}
\qarrow{B11}{B3}
\qarrow{B3}{A11}
\qarrow{A11}{A3}

\qarrow{C21}{C4}
\qarrow{C4}{B21}
\qarrow{B21}{B4}
\qarrow{B4}{A21}
\qarrow{A21}{A4}

\draw (6,2.6) node{$\vdots$};
\draw (8,2.6) node{$\vdots$};
\draw (10,2.6) node{$\vdots$};
%
{\color{blue}
\draw (7.5,9.1) node{$\vdots$};

\draw (7.2,8.9) circle(2pt) coordinate(E1) node[right]{$v^4_{n-1}$};
\draw (7.2,7.9) circle(2pt) coordinate(E2) node[right]{$v^4_{n-\frac{1}{2}}$};
\draw (7.2,6.9) circle(2pt) coordinate(E3) node[right]{$v^4_{n}$};
\draw (7.2,5.9) circle(2pt) coordinate(E4) node[right]{$v^4_{n+\frac{1}{2}}$};
\draw (7.2,4.9) circle(2pt) coordinate(E5) node[right]{$v^4_{n+1}$};
\draw (7.2,3.9) circle(2pt) coordinate(E6) node[right]{$v^4_{n+\frac{3}{2}}$};

\draw (7.5,3) node{$\vdots$};

\qarrow{E1}{E2}
\qarrow{E2}{E3}
\qarrow{E3}{E4}
\qarrow{E4}{E5}
\qarrow{E5}{E6}

\qarrow{E1}{A4}
\qarrow{E2}{A11}
\qarrow{E3}{B4}
\qarrow{E4}{B11}
\qarrow{E5}{C4}
\qarrow{E6}{C11}

\qarrow{A11}{E1}
\qarrow{B4}{E2}
\qarrow{B11}{E3}
\qarrow{C4}{E4}
\qarrow{C11}{E5}

}
\end{scope}
\end{tikzpicture}
\caption{The quiver $Q(F_4)$.}
\label{TQF4}
\end{figure}

\paragraph{\textbf{The case of $G_2$}}
We have $\mathbf{D}=(1,3)$, thus $d = 1$. 
A quiver $Q(G_2)$ of a vertex set $I = \{v_n^i;~ i \in S, n \in \Z \}$ is defined in the following way: 
\begin{itemize}
\item
we have arrows $v_n^1 \to v_{n+1}^{1}$ and  $v_n^2 \to v_{n+3}^{2}$,
for $n \in \Z$,
\item
we have an arrow $v_n^2 \to v_n^1$ and an arrow
$v_{n+3}^{1} \to v_n^{2}$, for $n \in \Z$. 
\end{itemize}
See Figure \ref{TQG2} for the quiver $Q(G_2)$.

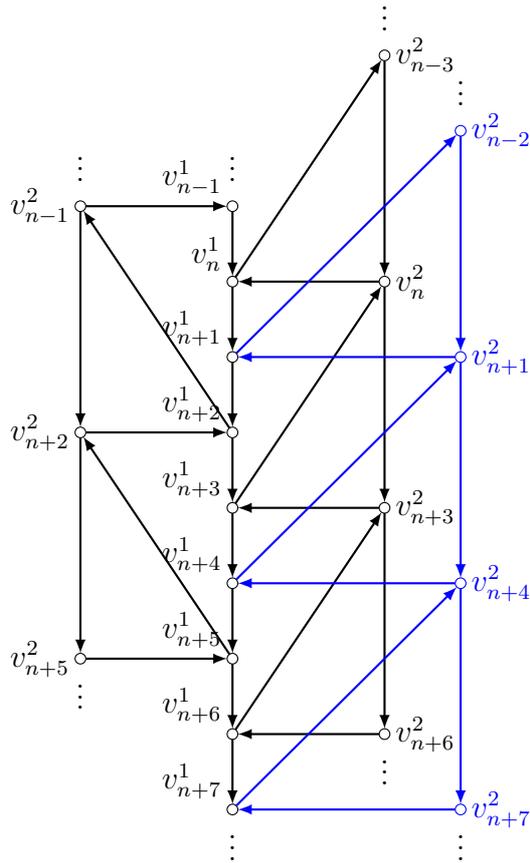
\begin{figure}[ht]
\begin{tikzpicture}
\begin{scope}[>=latex]
\draw (1,8.6) node{$\vdots$};
\draw (3,8.6) node{$\vdots$};
\draw (5,10.6) node{$\vdots$};
\draw (6,9.6) node{$\vdots$};

\draw (3,8) circle(2pt) coordinate(A1) node[above left]{$v^1_{n-1}$};
\draw (3,7) circle(2pt) coordinate(A2) node[above left]{$v^1_{n}$};
\draw (3,6) circle(2pt) coordinate(A3) node[above left]{$v^1_{n+1}$};
\draw (3,5) circle(2pt) coordinate(A4) node[above left]{$v^1_{n+2}$};
\draw (3,4) circle(2pt) coordinate(A5) node[above left]{$v^1_{n+3}$};
\draw (3,3) circle(2pt) coordinate(A6) node[above left]{$v^1_{n+4}$};
\draw (3,2) circle(2pt) coordinate(A7) node[above left]{$v^1_{n+5}$};
\draw (3,1) circle(2pt) coordinate(A8) node[above left]{$v^1_{n+6}$};
\draw (3,0) circle(2pt) coordinate(A9) node[above left]{$v^1_{n+7}$};

\qarrow{A1}{A2}
\qarrow{A2}{A3}
\qarrow{A3}{A4}
\qarrow{A4}{A5}
\qarrow{A5}{A6}
\qarrow{A6}{A7}
\qarrow{A7}{A8}
\qarrow{A8}{A9}

\draw (5,10) circle(2pt) coordinate(B1) node[right]{$v^2_{n-3}$};
\draw (5,7) circle(2pt) coordinate(B2) node[right]{$v^2_{n}$};
\draw (5,4) circle(2pt) coordinate(B3) node[right]{$v^2_{n+3}$};
\draw (5,1) circle(2pt) coordinate(B4) node[right]{$v^2_{n+6}$};

\qarrow{B1}{B2}
\qarrow{B2}{B3}
\qarrow{B3}{B4}

\qarrow{B2}{A2}
\qarrow{B3}{A5}
\qarrow{B4}{A8}

\qarrow{A2}{B1}
\qarrow{A5}{B2}
\qarrow{A8}{B3}

{\color{blue}
\draw (6,9) circle(2pt) coordinate(C1) node[right]{$v^2_{n-2}$};
\draw (6,6) circle(2pt) coordinate(C2) node[right]{$v^2_{n+1}$};
\draw (6,3) circle(2pt) coordinate(C3) node[right]{$v^2_{n+4}$};
\draw (6,0) circle(2pt) coordinate(C4) node[right]{$v^2_{n+7}$};

\qarrow{C1}{C2}
\qarrow{C2}{C3}
\qarrow{C3}{C4}

\qarrow{C2}{A3}
\qarrow{C3}{A6}
\qarrow{C4}{A9}

\qarrow{A3}{C1}
\qarrow{A6}{C2}
\qarrow{A9}{C3}

}

\draw (1,8) circle(2pt) coordinate(D1) node[left]{$v^2_{n-1}$};
\draw (1,5) circle(2pt) coordinate(D2) node[left]{$v^2_{n+2}$};
\draw (1,2) circle(2pt) coordinate(D3) node[left]{$v^2_{n+5}$};

\qarrow{D1}{D2}
\qarrow{D2}{D3}

\qarrow{D1}{A1}
\qarrow{D2}{A4}
\qarrow{D3}{A7}

\qarrow{A4}{D1}
\qarrow{A7}{D2}

\draw (1,1.6) node{$\vdots$};
\draw (3,-0.4) node{$\vdots$};
\draw (5,0.6) node{$\vdots$};
\draw (6,-0.4) node{$\vdots$};

\end{scope}
\end{tikzpicture}
\caption{The quiver $Q(G_2)$}
\label{TQG2}
\end{figure}

\begin{remark}\label{rem:HL}
The infinite quivers $Q(\mg)$ was introduced in \cite{HL16},
where they used its semi-infinite subquiver to study the cluster structure of the $q$-character of Kirillov-Reshetikhin modules. They also used the full quiver to study the category of representations of a Borel subalgebra in \cite{HL16b}.  
For $\mg = A_\ell$ and other simply-laced-cases, periodic quivers $Q_m(\mg)$ appeared in \cite{ILP16} and \cite{IIO19}.  
\end{remark}

\subsection{The periodic quiver $Q_m(\mg)$ and Weyl group action}

Let $m$ be an integer bigger than one.
We consider the case that the vertex set $I$ of the quiver $Q(\mg)$ is `$m$-periodic'; $I = I_m = \{v_n^i;~ i \in S, n \in d \Z/m d'\Z \}$.
We write $Q_m(\mg)$ for this quiver.

By construction, in $Q_m(\mg)$ there are oriented circles $P_{i,\gamma_i}; ~i \in S, \gamma_i = 1,2,\ldots,d_i/d$.
When $d_i/d = 1$, we write $P_i$ for $P_{i,1}$ for simplicity. 
We present the circles case by case in the following.
\\

\paragraph{\textbf{The case of $A_\ell, D_\ell$ and $E_{6,7,8}$}}
In the cases of $\mg$ of simply-laced Dynkin diagram, we have $\ell$ circles as
\begin{align}
\label{PathA}
P_i : v^i_1 \to v^i_2 \to \cdots \to v^i_m \to v^i_1
\end{align}
for $i \in S$.
\\

\paragraph{\textbf{The case of $B_\ell$}}
The vertex set of $Q_m(B_\ell)$ is $I = \{v^i_n;~ i \in S, n \in \frac{1}{2} \Z / m \Z \}$ and we have $2\ell-1$ circles in total:
\begin{align}
\label{PathB1}
\begin{split}
&P_{i,1} : v^i_1 \to v^i_2 \to \cdots \to v^i_m \to v^i_1,
\\
&P_{i,2} : v^i_{\frac{1}{2}} \to v^i_{\frac{3}{2}} \to \cdots \to v^i_{m-\frac{1}{2}} \to v^i_{\frac{1}{2}}
\end{split}
\end{align}
for $i=1,\ldots,\ell-1$, and
\begin{align}
\label{PathB2}
P_\ell : v^\ell_{\frac{1}{2}} \to v^\ell_1 \to v^\ell_{\frac{3}{2}} \to v^\ell_2 \to \cdots \to v^\ell_{m-\frac{1}{2}} \to v^\ell_m \to v^\ell_{\frac{1}{2}}.
\end{align} 

\paragraph{\textbf{The case of $C_\ell$}}
The vertex set of $Q_m(C_\ell)$ is $I = \{v^i_n;~ i \in S, n \in \Z / 2m \Z \}$, and we have $\ell+1$ circles:
\begin{align}
\label{PathC1}
P_i : v^i_1 \to v^i_2 \to \cdots \to v^i_{2m} \to v^i_1
\end{align}
for $i=1,\ldots,\ell-1$, and 
\begin{align}
\label{PathC2}
&P_{\ell,k} : v^\ell_k \to v^\ell_{k+2} \to \cdots \to v^\ell_{k+2m-2} \to v^\ell_k; ~k=1,2.
\end{align}

\paragraph{\textbf{The case of $F_4$}}
The vertex set of $Q_m(F_4)$ is $I = \{v^i_n;~ i \in S, n \in \frac{1}{2} \Z / m \Z \}$, and we have $6$ circles:
$P_{i,1}$ and $P_{i,2}$ as \eqref{PathB1} for $i=1,2$, and $P_i$ as \eqref{PathB2} for $i=3,4$.
\\

\paragraph{\textbf{The case of $G_2$}}
The vertex set of $Q_m(G_2)$ is $I = \{v^i_n;~ i \in S, n \in \Z / 3m \Z \}$, and we have $4$ circles:
\begin{align*}
&P_1 : v^1_1 \to v^1_2 \to \cdots \to v^1_{3m} \to v^1_1
\\
&P_{2,k} : v^2_k \to v^2_{k+3} \to v^2_{k+6} \to \cdots \to v^2_{k+3m-3} \to v^2_k; ~k=1,2,3.
\\ \nonumber
\end{align*}

For an oriented circle $P$ of length $p$ as $v_1 \to v_2 \to \cdots \to v_p \to v_1$, define a sequence of mutation $M(P) := \mu_{p-2} \cdots \mu_2 \mu_1$ and 
\begin{align}\label{eq:R-op}
  R(P) := M(P)^{-1} \circ (v_{p-1}, v_p) \circ \mu_p \mu_{p-1} \circ M(P).
\end{align}
Here $\mu_k$ is the mutation at vertex $v_k$, and $(v_{p-1}, v_p) \in \mathfrak{S}_p$ interchanges vertices 
$v_{p-1}$ and $v_p$.
Also define a polynomial $f_X(P,v_n)$ in the $X$-variables on $P$:
$$
  f_X(P,v_n) := 1 + \sum_{k=0}^{p-2} X_n X_{n-1} \cdots X_{n-k}; ~ v_n \in P.
$$

\begin{prop}
The mutation sequence $R(P_{i,\gamma_i})$ preserves the quiver $Q_m(\mg)$.
\end{prop}

This is obtained as a corollary of \cite[Theorem 7.7]{GS16}.

\begin{thm}\label{R-action}
The action of $R(P_{i,\gamma_i})$ on the seed $(Q_m(\mg),\mathbf{X},\mathbf{A})$ induces the following action $R(P_{i,\gamma_i})^\ast$ on $\C(\mathbf{X})$ and $\C(\mathbf{A})$.
\begin{enumerate}
\item The action on $\C(\mathbf{X})$ when $d_i \geq d_j$ is   
\begin{align}\label{eq:RonX}
R(P_{i,\gamma_i})^\ast (X^j_n) =  
  \begin{cases}
    \displaystyle{\frac{f_X(i,n)}{X^j_{n-d_i} \,f_X(i,{n-2d_i})}}  
    & i=j, ~v^i_n \in P_{i,\gamma_i},
    \\[2mm]
    \displaystyle{X^{j}_n \frac{X^i_{n-d_i} f_X(i,{n-2 d_i})}{f_X(i,{n-d_i})}} 
    & v^j_n \leftarrow v^i_n, ~ v^i_n \in P_{i,\gamma_i},
    \\[2mm]    
    \displaystyle{X^{j}_n \frac{X^i_{n} f_X(i,{n-d_i})}{f_X(i,{n})}}
    & v^j_n \to v^i_n, ~ v^i_n \in P_{i,\gamma_i},
    \\[2mm]    
    \displaystyle{X^{j}_n \frac{X^i_{n-\frac{1}{2}} f_X(i,{n-d_i-\frac{1}{2}})}{f_X(i,{n-\frac{1}{2}})}}
    & v^j_n \to v^i_{n-\frac{1}{2}}, ~ v^i_{n-\frac{1}{2}} \in P_{i,\gamma_i},
    \\[1mm]
    X^j_n & \text{otherwise},
  \end{cases}
\end{align}  
where we write $f_X(i,n)$ for $f_X(P_{i,\gamma_i};v^i_n)$ for $v^i_n \in P_{i,\gamma_i}$ without confusion.
When $d_i < d_j$, non-trivial actions are
\begin{align}\label{eq:RonX-exc}
R(P_i)^\ast (X^j_n) 
= 
\begin{cases}
\displaystyle{X^j_n \frac{X^i_{n-\frac{1}{2}} X^i_{n} f_X(i,n-1)}{f_X(i,n)}}
 \quad (i,j)= \begin{cases} (\ell,\ell-1); B_\ell, \\ (3,2); F_4, \end{cases}
\\
\displaystyle{X^\ell_n \frac{X^{\ell-1}_{n} X^{\ell-1}_{n+1} f_X(\ell-1,n-1)}{f_X(\ell-1,n+1)}} \quad (i,j) = (\ell-1,\ell); C_\ell,
\\[2mm]
\displaystyle{X^2_n \frac{X^{1}_{n} X^{1}_{n+1} X^{1}_{n+2} f_X(1,n-1)}{f_X(1,n+2)}} 
  \quad (i,j) = (1,2); G_2.
\end{cases}
\end{align}
\item The action on $\C(\mathbf{A})$ is  
\begin{align}\label{A-transf}
  R(P_{i,\gamma_i})^\ast (A^j_n) = 
  \begin{cases}
  f_A(i,\gamma_i) A^j_n & i=j, ~v^i_n \in P_{i,\gamma_i},
  \\ 
  A^j_n & \text{otherwise},
  \end{cases}
\end{align}  
where 
$$
  f_A(i,\gamma_i)  
  = 
  \sum_{n: v_n^i \in P_{i,\gamma_i}} \frac{1}{A^i_n A^i_{n+d_i}}
  \prod_{j:i \lhd j} A^j_n \cdot \prod_{j:i \rhd j} A^j_{n+d_i},
$$
for all cases,
except for $i=\ell-1,\ell$ for $B_\ell$, $i=\ell-1$ for $C_\ell$, $i=2,3$ for $F_4$, and $i=2$ for $G_2$. 
Here $j \lhd i$ means that $j < i$ and $C_{ij} \neq 0$.
For the exceptional cases, $f_A(i,\gamma_i)$ is set to be 
$$ 
  f_A(i,\gamma_i) = \displaystyle{\sum_{n:v^i_n \in P_{i,\gamma_i}} \frac{A^{i+1}_{n+\frac{1}{2}} A^{i-1}_{n+1}}{A^i_n A^i_{n+1}}} \quad i=\ell-1; B_\ell,~ i=2; F_4
$$
$$ 
  f_A(i,1) 
  =
  \begin{cases}
  \displaystyle{\sum_{n: v_n^\ell \in P_\ell} \frac{A^{\ell-1}_{n+\frac{1}{2}} A^{\ell-1}_{n}}{A^\ell_n A^\ell_{n+\frac{1}{2}}}} & i=\ell; B_\ell,
  \\
  \displaystyle{\sum_{n: v_n^{\ell-1} \in P_{\ell-1}} 
\frac{A^{\ell-2}_{n+1} A^\ell_n A^\ell_{n-1}}{A^{\ell-1}_n A^{\ell-1}_{n+1}}} & i=\ell-1; C_\ell,  
  \\
  \displaystyle{\sum_{n: v_n^{3} \in P_3} \frac{A^{2}_{n+\frac{1}{2}} A^{2}_{n}A^4_n}{A^3_n A^3_{n+\frac{1}{2}}}} & i=3; F_4,
  \\
  \displaystyle{\sum_{n: v_n^1 \in P_1} 
\frac{A^2_n A^2_{n-1} A^2_{n-2}}{A^1_n A^1_{n+1}}} & i=1; G_2. 
  \end{cases}
$$
\end{enumerate}
\end{thm}

This theorem follows from the next lemma.

\begin{lem}
Let $Q$ be a quiver including an oriented circle $P$ as $v_1 \to v_2 \to \ldots \to v_p \to v_1$ and a vertex $v_a$ with two arrows $v_a \to v_1$ and $v_k \to v_a$ for some $1 \leq k \leq p$. Then a mutation sequence $R(P)$ preserves $Q$ and 
induces the action on $\C(\mathbf{X})$ and $\C(\mathbf{A})$ as follows:
\begin{align*}
  &R(P)^\ast(X_n) = 
  \begin{cases}
    \displaystyle{\frac{f_X(P,v_n)}{X_{n-1} \, f_X(P,v_{n-2})}} & v_n \in P,
    \\[2mm]
    \displaystyle{X_a X_1 X_2 \ldots X_{k-1} \frac{f_X(P,v_p)}{f_X(P,v_{k-1})}} & n = a, 
  \end{cases}
  \\
  &R(P)^\ast(A_n) =
  \begin{cases}
    A_n f_A(P) & v_n \in P,
    \\
    A_a  & n = a. 
  \end{cases}  
\end{align*}
Here we define 
\begin{align}
  &f_A(P) := \sum_{j=1}^{k-1} \frac{A_a}{A_j A_{j+1}} + \sum_{j=k}^{p} \frac{1}{A_j A_{j+1}}. 
\end{align}
\end{lem}

\begin{proof}
The claim for $X_n$ of $v_n \in P$ is well known (see for example \cite{Bu14,ILP16,GS16}). The claims for $X_a$ and $A_n$ are essentially shown in \cite[Theorem 7.7]{GS16}, for which we present the proof for self-consistency.

The case of $X_n = X_a$ is shown 
by induction in the same manner as \cite{MOT19}. 
When $p = 2$, it is easy to check the claim.
Assume the claim for $p \geq L-1$.
When $p = L$, note that a quiver $Q' := \mu_1(Q)$ contains a smaller oriented circle $P': v_2 \to v_3 \to \cdots \to v_L \to v_2$ and $v_a$ is connetced with $P'$ by two arrows $v_a \to v_2$ and $v_k \to v_a$.  
We factorize the sequense of mutation $R(P)$ as $R(P) = \mu_1 \circ R(P') \circ \mu_1$, and consider the sequence of $X$-seeds: 
$$
  (Q, \mathbf{X}) \stackrel{\mu_1}{\mapsto} (Q', \mathbf{X}') \stackrel{R(P')}{\mapsto} (Q', \mathbf{X}'') 
\stackrel{\mu_1}{\mapsto} (Q, \mathbf{X}''').
$$
By the assumption, we have
$$
  X_a'' = R(P')^\ast(X_a') = X_a' X_2' X_3' \cdots X_{k-1}' \frac{f_{X'}(P',v_L)}{f_{X'}(P',v_{k-1})}.
$$
On the other hand, from the above sequence of seeds we obtain
$$
  X_i' = 
  \begin{cases}
    X_1^{-1} & i=1,
    \\
    X_2(1+X_1) & i=2,
    \\
    X_i(1+X_1^{-1})^{-1} & i=L,a,
    \\
    X_i & \text{otherwise},
  \end{cases}
  \qquad 
  X_1'' = (X_1''')^{-1} = \frac{X_L f_X(P,v_{L-1})}{f_X(P,v_1)}.
$$  
Thus by direct calculations, we obtain $f_{X'}(P',v_L)=f_X(P,v_1)/(1+X_1)$, $f_{X'}(P',v_{k-1}) = f_X(P,v_{k-1})$, and   
\begin{align*}
  X_a''' &= X_a''(1+X_1'') 
  \\
  &= X_a (1+X_1^{-1})^{-1} X_2(1+X_1) X_3 \cdots X_{k-1} \frac{f_X(P,v_1)}{(1+X_1)f_X(P,v_{k-1})} \left(1+\frac{X_L f_X(P,v_{L-1})}{f_X(P,v_1)}\right)
  \\
  &= X_a X_1 X_2 \cdots X_{k-1} \frac{f_X(P,v_L)}{f_X(P,v_{k-1})},
\end{align*}
which completes the induction.

The claim for $A_n$ of $v_n \in P$ is shown by slightly modifying the way in \cite{ILP16}. Consider a sequence of $A$-seeds:
$$
  (Q, \mathbf{A}) \stackrel{M(P)}{\mapsto} (Q[p-2], \mathbf{A}[p-2])
  \stackrel{\mu_{p} \mu_{p-1}}{\mapsto}(Q[p], \mathbf{A}[p]).
$$  
We inductively obtain that 
$$
  A_n[p-2] = \frac{A_{n+1}}{A_1} + A_a \sum_{j=1}^{k-1} \frac{A_{n+1} A_p}{A_j A_{j+1}} + \sum_{j=k}^{p-2} \frac{A_{n+1} A_p}{A_j A_{j+1}}; ~ n=1,2,\ldots,p-2,
$$
thus we have $A_{p-1}[p] = \frac{A_{p-2}[p-2]+1}{A_{p-1}} = A_p f_A(P) = R(P)^\ast (A_p)$, and also $A_{p}[p] = A_{p-1} f_A(P) = R(P)^\ast (A_{p-1})$. 
Due to the translation invariance of the action of $R(P)$, the claim follows. 
\end{proof}

Remark that when $d_i/d > 1$, $R(P_{i,\gamma_i})$ for $\gamma_i = 1,\cdots,d_i/d$ are mutually commutative.
We define $R_i := R(P_{i,1}) \circ R(P_{i,2}) \circ \cdots \circ R(P_{i,d_i/d})$ for $i \in S$.

\begin{prop}
For a seed $(Q(\mg),\mathbf{X},\mathbf{A})$, let $p$ be the positive map from $\C(\mathbf{A})$ to $\C(\mathbf{X})$ \cite{FZ-CA4} given by 
\begin{align}\label{eq:positivemap}
  p^\ast (X^i_n) = \prod_{v^j_m \in Q_m(\mg)} (A^j_m)^{\ve_{v^i_n v^j_m}}.
\end{align}
Then $R_i$ belongs to the peripheral subgroup of $\Gamma_{Q_m(\mg)}$, i.e.
$(R_i)^\ast p^\ast (X^i_n) = p^\ast (X^i_n)$ for any $X^i_n$. 
\end{prop}

\begin{proof}
The simply-laced cases are already shown in \cite[Lemma 3.12]{IIO19}.
We give an alternative proof which works for general $\mg$.
Note that $\xi^i_n := A^i_{n-d_i}/A^i_n \in \C(\mathbf{A})$ is invariant under the action of any $R_j$. Write $\C(\xi)$ for the subfield of $\C(\mathbf{A})$ generated by the $\xi^i_n$. 
It holds that $p^\ast (X^i_n)$ belongs to $\C(\xi)$, since 
for a circle $P_{j,k}$ in the quiver $Q_m(\mg)$ and the vertex $v^i_n$ we have only two possibilities:
\begin{enumerate}
\item[(i)]
there is no arrow among $P_{j,k}$ and $v^i_n$,
\item[(ii)]
there are only two arrows among $P_{j,k}$ and $v^i_n$ such that 
$v^j_m \to v^i_n$ and $v^j_{m'} \leftarrow v^i_n$ for some $v^j_m, v^j_{m'} \in P_{j,k}$.
\end{enumerate}   
Thus the claim follows.  
\end{proof}

From Theorem \ref{R-action} the following is obtained.
\begin{cor}\label{cor:tropR}
Let $\mathbb{P}_\trop(\mathbf{u})$ be the tropical semifield of rank $|I|$, where $\mathbf{u} = (u_n^i; v_n^i \in I)$. For $i \in S$ and a tropical $X$-seed $(Q_m(\mg), \mathbf{x})$ where all $x^i_n$ for $n \in d \Z / d'm \Z$ have positive tropical sign, the tropical action $R^\trop_i$ is expressed as follows:
when $d_i \geq d_j$,
\begin{align}\label{eq:trop-R1}
(R^\trop_i)^\ast (x_n^j) = 
\begin{cases}
(x^i_{n-d_i})^{-1} & j = i, 
\\
x^j_{n} x^i_{n-d_i} & v^j_n \leftarrow v^i_n,
\\
x^j_{n} x^i_{n} & v^j_n \rightarrow v^i_n,
\\
x^j_{n} x^i_{n-\frac{1}{2}} & v^j_n \rightarrow v^i_{n-\frac{1}{2}},
\\
x^j_n & otherwise,
\end{cases}
\end{align}
and when $d_i < d_j$, 
\begin{align}\label{eq:trop-R2}
(R^\trop_i)^\ast (x^j_n) =
\begin{cases}
x^j_n x^i_{n-\frac{1}{2}} x^i_{n} & (i,j) =
\begin{cases} (\ell,\ell-1); B_\ell, \\ (3,2); F_4, \end{cases}
\\
x^\ell_n x^{\ell-1}_{n} x^{\ell-1}_{n+1} & (i,j) = (\ell-1,\ell); C_\ell,
\\
x^2_n x^{1}_{n} x^{1}_{n+1} x^{1}_{n+2} & (i,j) = (1,2); G_2,
\\
x^j_n & otherwise.
\end{cases}
\end{align}
\end{cor}

Finally we get to the realization of the Weyl group $W(\mg)$:

\begin{thm}\label{thm:R-Weyl}
\begin{enumerate}
\item[(1)]
The operators $R_i; ~i \in S$ generate an action of the Weyl group $W(\mg)$ on the seed. Precisely, we have $(R_i R_j)^{m_{ij}} = 1$ for $i,j \in S$.
\item[(2)]
Let $s_{i_p} s_{i_{p-1}} \cdots s_{i_1}$ be a reduces expression of $w \in W(\mg)$. Then $R(w) := R_{i_p} R_{i_{p-1}} \cdots R_{i_1}$ is a green sequence for $Q_m(\mg)$. When $w$ is the longest element in $W(\mg)$, then $R(w)$ is the maximal green
sequence for $Q_m(\mg)$. 
\end{enumerate}
\end{thm}
\begin{proof}
This claim in the case of simply-laced $\mg$ appeared in \cite{IIO19}.
We prove the other cases.
Let $\mathbb{P} = \mathbb{P}_\trop(\mathbf{u})$ be the tropical semifield of rank $|I|$. In this proof we write $Q$ for $Q_m(\mg)$.
\\
(1) Due to Theorem \ref{thm:periodicity}, it is satisfactory to check the claim for the tropical $X$-seed in $\mathbb{P}$.
The case of $(i,j)$ as  $C_{ij} C_{ji}=1$ (i.~e. $m_{ij} = 3$) is done in \cite{ILP16}. The other cases are checked case by case using Corollary \ref{cor:tropR}, and we shall demonstrate the case of $(d_1,d_2) = (1,\frac{1}{2})$ for $B_2$ and $(d_1,d_2) = (1,3)$ for $G_2$.  In the first case, for $\mathbf{x} = (x^i_n) \in \mathcal{X}_Q^+(\mathbb{P})$ we have 
\begin{align*}
&(x^1_n,x^2_n)_{n \in \frac{1}{2}\Z/m \Z} \stackrel{R_1^\ast}{\longmapsto} ((x^1_{n-1})^{-1}, x^2_n x^1_{n-\frac{1}{2}})
\\
&\stackrel{R_2^\ast}{\longmapsto} ((x^1_{n-1})^{-1} x^2_{n} x^1_{n-\frac{1}{2}} x^2_{n-\frac{1}{2}} x^1_{n-1}, (x^2_{n-\frac{1}{2}} x^1_{n-1})^{-1}) 
= (x^2_n x^2_{n-\frac{1}{2}} x^1_{n-\frac{1}{2}}, (x^2_{n-\frac{1}{2}} x^1_{n-1})^{-1})
\\
&\stackrel{R_1^\ast}{\longmapsto}
((x^2_{n-1}x^2_{n-\frac{3}{2}} x^1_{n-\frac{3}{2}})^{-1}, (x^2_{n-\frac{1}{2}} x^1_{n-1})^{-1}x^2_{n-\frac{1}{2}}x^2_{n-1} x^1_{n-1} )
= ((x^2_{n-1}x^2_{n-\frac{3}{2}} x^1_{n-\frac{3}{2}})^{-1}, x^2_{n-1})
\\
&\stackrel{R_2^\ast}{\longmapsto}
((x^2_{n-1}x^2_{n-\frac{3}{2}} x^1_{n-\frac{3}{2}})^{-1} x^2_{n-1} x^2_{n-\frac{3}{2}}, (x^2_{n-\frac{3}{2}})^{-1}) = ((x^1_{n-\frac{3}{2}})^{-1}, (x^2_{n-\frac{3}{2}})^{-1})_{n \in \frac{1}{2}\Z/m \Z}.
\end{align*} 
This coincides with $(R_2 R_1R_2 R_1)^\ast (x^1_n,x^2_n)_{n \in \Z/2}.$
In the second case of $G_2$, we have
\begin{align*}
&(x^1_n,x^2_n)_{n \in \Z/3m \Z} \stackrel{R_1^\ast}{\longmapsto} 
((x^1_{n-1})^{-1}, x^2_n x^1_n x^1_{n+1} x^1_{n+2})
\\
&\stackrel{R_2^\ast}{\longmapsto}
((x^1_{n-1})^{-1} x^2_{n-3} x^1_{n-3} x^1_{n-2} x^1_{n-1}, (x^2_{n-3} x^1_{n-3} x^1_{n-2} x^1_{n-1})^{-1})  
\\
& \qquad \quad = (x^2_{n-3} x^1_{n-3} x^1_{n-2}, (x^2_{n-3} x^1_{n-3} x^1_{n-2} x^1_{n-1})^{-1})
\\
&\stackrel{R_1^\ast}{\longmapsto} 
((x^2_{n-4} x^1_{n-4} x^1_{n-3})^{-1}, x^2_{n-2} x^1_{n-2} x^1_{n-1} x^2_{n-1} x^1_{n})
\\
&\stackrel{R_2^\ast}{\longmapsto}
(x^2_{n-5} x^1_{n-5}, (x^2_{n-5} x^1_{n-5} x^1_{n-4} x^2_{n-4} x^1_{n-3})^{-1})\\\
&\stackrel{R_1^\ast}{\longmapsto}
((x^2_{n-6} x^1_{n-6})^{-1}, x^2_{n-3}) 
\stackrel{R_2^\ast}{\longmapsto}
((x^1_{n-6})^{-1}, (x^2_{n-6})^{-1})_{n \in \Z/3m\Z},
\end{align*}
which coincides with $(R_2 R_1 R_2 R_1 R_2 R_1)^\ast (x^1_n,x^2_n)_{n \in \Z/3m \Z}$. 
\\
(2) It is known that each $R_i$ is green \cite{Bu14} (also see the proof of \cite[Lemma 7.5]{ILP16}).
Let $L(\mg) = \oplus_{i \in S} \Z \alpha_i$ be the root lattice,
and let $L(\mg)_+$ be its nonnegative part, i.~e. $L(\mg)_+ = \oplus_{i \in S} \Z_{\geq 0} \alpha_i$. 
Define an embedding $\phi : L(\mg) \to \text{Fun}(\mathcal{X}_{Q}^+(\mathbb{P}))$ given by 
$$
  \sum_{i \in S} n_i \alpha_i \mapsto \prod_{i \in S} \prod_{n \in d\Z/md'\Z} (x^i_n)^{n_i}.
$$
For $f \in \text{Fun}(\mathcal{X}_{Q}^+(\mathbb{P}))$, we write $f(\mathbf{\xi})$ for the evaluation of $f$ at $\xi \in \mathcal{X}_{Q}^+(\mathbb{P})$.
Remark that $v \in L(\mg)$ belongs to $L(\mg)_+$ if and only if $\phi(v)(\mathbf{u})$ has positive tropical sign.

From Corollary \ref{cor:tropR}, we obtain 
$$
  (R^\trop_i)^\ast(\phi(\alpha_j)) = 
  \begin{cases}
    \phi(\alpha_i)^{-1} & j=i,
    \\
    \phi(\alpha_j) \phi(\alpha_i)^{-C_{ij}} & \text{otherwise}.   
  \end{cases}
$$
For a reduced expression $s_{i_p} s_{i_p-1} \cdots s_{i_1}$ of $w \in W(\mg)$, write $w_{(j)}$ for $s_{i_j} \cdots s_{i_2} s_{i_1}$ for $j=1,\ldots,p$.
Then, for $j = 1,\ldots,p-1$ the followings are proved:
\begin{enumerate}
\item[(i)] 
$\phi(w_{(j)}^{-1}\alpha_i)(\mathbf{u}) = \phi(\alpha_i)(R^\trop(w_{(j)})(\mathbf{u}))$,
\item[(ii)] 
$R^\trop(w_{(j)})(\mathbf{u})^{s_{j+1}}_n$ has positive tropical sign for all $n \in d \Z/d'm\Z$,   
\end{enumerate}
in the same manner as \cite[Proposition 3.17]{IIO19}.
It follows that the embedding $\phi$ is $W(\mg)$-equivalent, i.~e.
$R^\trop(w)^\ast (\phi(v)) = \phi(w^{-1}v)$ for $v \in L(\mg)$ and $w \in W(\mg)$.

From (ii), it follows that $R(w)$ is a green sequence.
When $w$ is the longest element in $W(\mg)$, we have $w \alpha_i = -\alpha_{i^\ast}$ for $i \in S$, where $i^\ast$ is the image of $i$ by the Dynkin involution on $S$. Hence $\phi(w \alpha_i)(\mathbf{u})$ has negative tropical sign, and it follows that $R^\trop(w)(\mathbf{u}) \in \mathcal{X}^-_Q(\mathbb{P})$.
\end{proof}

\begin{remark}
The mutation sequence $R(P)$ \eqref{eq:R-op} first appeared in \cite{Bu14} in studying maximal green sequence for an oriented circle $P$. It is also applied to study the cluster realization of the geometric $R$-matrices \cite{ILP16}, the cluster realizaion of the Weyl groups \cite{IIO19}, and the higher Teichm\"uller theory \cite{GS16,IIO19}.
When $\mg$ has non-simply-laced Dynkin diagram, Theorem \ref{thm:R-Weyl} gives another cluster realization of the Weyl group than \cite{IIO19}. 
\end{remark}

\subsection{Remark on the case of affine Lie algebras}
\label{subsec:affine}

We naturally extend the infinite quiver $Q(\mg)$ to that for the corresponding   non-twisted affine Lie algebra $\hat{\mg}$. 
For the simply-laced cases $\hat{\mg} = A_\ell^{(1)}, D_\ell^{(1)}, E_{6,7,8}^{(1)}$, $Q(\hat{\mg})$ is defined in the same way as \cite{IIO19}.
We briefly explain the other cases here. 
  
Define $\hat{S} := S \cup \{0\}$, and let $\hat{I}:= \{v^i_n; ~i \in \hat{S}, n \in d\Z\}$ be the vertex set of $Q(\hat{\mg})$.
In the case of $B_\ell^{(1)}$ of $\ell \geq 3$, we add vertices 
$\{v^0_n; ~ n \in \frac{1}{2} \Z \}$ to $Q(B_\ell)$ with arrows 
$$
  v^0_n \to v^0_{n+1}, ~~ v^2_n \to v^0_n, ~~ v^0_n \to v^2_{n-1}; ~n \in \frac{1}{2}\Z.
$$
In the case of $C_\ell^{(1)}$ of $\ell \geq 2$, we add vertices 
$\{v^0_n; ~ n \in \Z \}$ to $Q(C_\ell)$ with arrows 
$$
  v^0_n \to v^0_{n+2}, ~~ v^1_n \to v^0_n, ~~ v^0_n \to v^1_{n-2}; ~n \in \Z.
$$
When $\ell=2$, we identify $Q(B_2^{(1)})$ with $Q(C_2^{(1)})$.
In the case of $F_4^{(1)}$, we add vertices 
$\{v^0_n; ~ n \in \frac{1}{2} \Z \}$ to $Q(B_\ell)$ with arrows 
$$
  v^0_n \to v^0_{n+1}, ~~ v^1_n \to v^0_n, ~~ v^0_n \to v^1_{n-1}; ~n \in \frac{1}{2}\Z.
$$
Finally, in the case of $G_2^{(1)}$ we add vertices 
$\{v^0_n; ~ n \in \Z \}$ to $Q(G_2)$ with arrows 
$$
  v^0_n \to v^0_{n+3}, ~~ v^2_n \to v^0_n, ~~ v^0_n \to v^2_{n-3}; ~n \in \Z.
$$

We further define the $m$-periodic quiver $Q_m(\hat{\mg})$ of the vertex set 
$\hat{I} = \{v^i_n; ~i \in \hat{S}, n \in d\Z/d'm \}$, in the same manner as $Q_m(\mg)$. In $Q_m(\hat{\mg})$ we have new circles $P_{0,k}$ and mutation sequences $R(P_{0,k})$ for $k = 1,\ldots, d'/d$. 
Then the generators of the Weyl group $W(\hat{\mg}) \subset \Gamma_{Q_m(\hat{\mg})}$ are obtained by adding $R_0 := \prod_{k=1}^{d'/d} R(P_{0,k})$ to those of $W(\mg)$.

Remark that for the dual $\mg^\vee$ of $\mg$, whose Cartan matrix is the transpose of $\mathbf{C}$, 
the infinite quivers, the $m$-periodic quivers and the Weyl group actions are defined similarly. This is the same for $\hat{\mg}^\vee$.

\begin{remark}
In [IIO19] it is introduced that the cluster realization of the Weyl groups for a symmetrizable Kac-Moody Lie algebras. This includes the cases of $\hat{\mg} = B_n^{(1)}, C_n^{(1)}, F_4^{(1)}, G_2^{(1)}$ with different periodic quivers from the above $Q_m(\hat{\mg})$.
\end{remark}

\section{Application to the $q$-characters}

After recalling the basic notions of the $q$-characters in \S \ref{subsec:q-ch}, we first present our results in the simplest case of $\mg = A_\ell$, and next explain the case of general $\mg$.

\subsection{$q$-characters of finite dimensional representations of $U_q(\hat{\mg})$}\label{subsec:q-ch}

We briefly recall the notion of the $q$-character and some related topics following \cite{FR98, FR99}.
We continue to use the notations in \S \ref{subsec:Lie-alg}.
Let $\hat{\mg}$ be the non-twisted affine Lie algebra corresponding to $\mg$,
and $U_q(\hat{\mg})$ be the quantum affine algebra.
For a while we assume that $q \in \C^\times$ is not a root of unity.
\\

\paragraph{\textbf{The $q$-characters}}
Write $\mathrm{Rep} \,U_q(\hat{\mg})$ for the Grothendieck ring of the category of the finite dimensional representation of $U_q(\hat{\mg})$ with homomorphisms of $U_q(\hat{\mg})$-modules.
The $q$-character is given as a ring homomorphism: 
$$
  \chi_q : \mathrm{Rep}\,U_q(\hat{\mg}) \to \mathbf{Y} := \Z[Y_{i,a_i}^{\pm 1}]_{i \in S, a_i \in \C^\times} \subset U_q(\tilde{\mathfrak{h}}[[z]]).
$$
See \cite[Section 3]{FR99} for the definition of $\tilde{\mathfrak{h}}$ and the explicit form of $Y_{i,a}$. 

For each $i \in S$ and $a \in \C^{\times}$, there is a unique irreducible representation $V_{\omega_i}(a) \in \mathrm{Rep} \,U_q(\hat{\mg})$, called the $i$th {\em fundamental representation} \cite{CP94}.
When we restrict this representation to $U_q(\mg) \subset U_q(\hat{\mg})$, it has the highest weight $\omega_i$. 
An important fact is that any irreducible finite-dimensional representation of $U_q(\hat{\mg})$ is isomorphic to a quotient of a submodule of the tensor product of the fundamental representations \cite[Corollary 1.4]{CP94}.
It is conjectured in \cite{FR99} and proved in \cite{FM99} that 
the image of $\chi_q$ is equal to
\begin{align}\label{eq:image-qch}
  \bigcap_{i \in S} \left(\Z[Y_{j,a}^{\pm 1}]_{j \neq i, a \in \C^\times} \otimes  \Z[Y_{i,b}(1+A_{i,bq_i}^{-1})]_{b \in \C^\times} \right),
\end{align}
where 
\begin{align}\label{q-A}
  A_{i,a} = Y_{i, a q_i} Y_{i,a q_i^{-1}}
  \prod_{j:C_{ji}=-1} Y_{j,a}^{-1}  
  \prod_{j:C_{ji}=-2} Y_{j,a q_j}^{-1} Y_{j,a q_j^{-1}}^{-1} 
  \prod_{j:C_{ji}=-3} Y_{j,a q_j^2}^{-1} Y_{j,a}^{-1} Y_{j,a q_j^{-2}}^{-1},
\end{align}
with $q_i = q^{d_i}$.

\begin{remark}\label{rem:q-root}
When $q$ is a root of unity $\ve$ as $\ve^n = 1$, the $q$-character map $\chi_q$ gives the character map $\chi_\ve : \mathrm{Rep} U_\ve^{\mathrm{res}}(\hat{\mg})\to \Z[Y_{i,a_i}^{\pm 1}]_{i \in S, a_i \in \C^\times}$ by setting $q$ to be $\ve$ \cite[Theorem 3.2]{FM01}. 
\end{remark}

We note that the rational functional fields $\C(Y_{i,a};i \in S, a \in \C^\times)$ and $\C(A_{i,a};i \in S, a \in \C^\times)$ have Poisson structure \cite[Appendix A]{FR98}:
\begin{align}
  \label{Poisson-A}
  &\{ A_{i,a}, A_{j,b} \}
  =
  {B}_{ij}(\frac{a}{b}) \, A_{i,a} A_{j,b},
  \\
  \label{Poisson-Y}
  &\{ Y_{i,a}, Y_{j,b} \}
  =
  {M}_{ij}(\frac{a}{b}) \, Y_{i,a} Y_{j,b},
\end{align}
where 
\begin{align*}
  &{B}_{ij}(z)
  =
  \sum_{m \in \mathbb{Z}} ~ {B}^{q^m}_{ij} z^m
  =
  \delta(q^{{B}_{ij}} z) - \delta(q^{-{B}_{ij}} z),
  \\
  &{M}_{ij}(z)
  =
  \sum_{m \in \mathbb{Z}} ~
  \bigl( \mathbf{D}^{q^m} (\mathbf{B}^{q^m})^{-1}
\mathbf{D}^{q^m}\bigr)_{ij} z^m.
\end{align*}
Here we set
$\delta(z) = \sum_{m \in \mathbb{Z}} \, z^m $, and matrices
$\mathbf{B}^q = ({B}^q_{ij})_{i,j \in
  \mathcal{I}}$
and
$\mathbf{D}^q = ({D}^q_{ij})_{i,j \in I}$
are defined as
\begin{align*}
  {B}^q_{ij}
  =
  q^{{B}_{ij}} - q^{-{B}_{ij}},
  \qquad
  {D}^q_{ij}
  =
  \delta_{ij} (q^{d_i} - q^{-d_i}).
  \\
\end{align*}

\paragraph{\textbf{Screening operators}}
Following \cite[\S 7]{FR99} we introduce the screening operators which appears in the $q$-deformation of the $\mathcal{W}$-algebras for $\mg$. 
For $i \in S$, define linear operators $\mathcal{S}_i$ given by
\begin{align}\label{eq:q-screening}
  \mathcal{S}_i : \mathbf{Y} \to 
                 \mathbf{YS}_i := \oplus_{b \in \C^\times} \mathbf{Y} \otimes S_{i,b};
               ~Y_{j,a} \mapsto \delta_{i,j} Y_{i,a} S_{i,a}
\end{align}
and the Leibniz rule: $\mathcal{S}_i \cdot (f g) = (\mathcal{S}_i \cdot f) g + f (\mathcal{S}_i \cdot g)$.
The operators $\mathcal{S}_i$ are called screening operators.
We consider the quotient $\mathbf{YS}'_i$ of $\mathbf{YS}_i$ with relations
\begin{align}\label{eq:S-A}
  S_{i,a q_i^2} = A_{i,aq_i} S_{i,a}.
\end{align} 

The $q$-deformed $\mathcal{W}$-algebra for $\mg$ is defined as the intersection of the kernels Ker $\mathcal{S}_i$ of $\mathcal{S}_i$ for $i \in S$.
The following is an important theorem which connects the $q$-characters and $q$-deformed $\mathcal{W}$-algebras.

\begin{thm}[Conjectured in \cite{FR99} and proved in \cite{FM99}]
The image of the character map $\chi_q$ equals the intersection of Ker $\mathcal{S}_i$ for $i \in S$.
\end{thm}

\begin{remark}
It is easy to see that $Y_{i,b}(1+A_{i,bq_i}^{-1})$ belongs to Ker $\mathcal{S}_i$, by using \eqref{q-A} and \eqref{eq:S-A}.
As the corollary of the above theorem, the $q$-deformed $\mathcal{W}$-algebra is generated by the $q$-characters of the fundamental representations $V_{\omega_i}(a)$.  
\end{remark}

\subsection{Correspondence of the $q$-deformation and the lattice: the case of $\mg = A_\ell$}\label{subsec:q-lattice-A}

For $a \in \C^\times / q^\Z$, define a rational functional field 
\begin{align}\label{eq:Y-def-A}
\mathbf{Y}_{a,q} := \C(Y_{i,a q^{2n+i-1}}; i \in S, n \in \Z).
\end{align}
Remark that the variables $A_{i,aq^{2n+i}}$ \eqref{q-A} belongs to $\mathbf{Y}_{a,q}$;
$$
  A_{i,aq^{2n+i}} 
  = \frac{Y_{i,a q^{2n+i-1}} Y_{i,a q^{2n+i+1}}}{Y_{i-1,a q^{2n+i}} Y_{i+1,a q^{2n+i}}}, 
$$
where $Y_{0, aq^{2n-1}} = Y_{\ell+1,a q^{2n+\ell}} = 1$ for $n \in \Z$.
Write $\mathbf{A}_{a,q}$ for the subfield of $\mathbf{Y}_{a,q}$ generated by 
$A_{i,aq^{2n+i}}; i \in S, n \in \Z$.

We introduce commuting variables $y_i(n)$ on a lattice $(i,n) \in S \times \Z$, and define a rational functional field
$\C(y) := \C(y_i(n); i \in S, n \in \Z)$. 
Further we define $a_i(n) \in \C(y)$ by
$$
  a_i(n) = \frac{y_i(n) y_{i}(n+1)}{y_{i-1}(n+1) y_{i+1}(n)} 
$$
for $i \in S$ and  $n \in \Z$, where we set 
$y_0(n) = y_{\ell+1}(n) = 1$ for $n \in \Z$. 
Let $\phi_a$ be an isomorphism of the rational functional fields given by
\begin{align}\label{qY-to-y}
  \phi_a : \mathbf{Y}_{a,q} \to \C(y); 
  ~ Y_{i,a q^{2n+i-1}} \mapsto y_i(n).
\end{align}
It is easy to see that $\phi_a(A_{i,a q^{2n+i}}) = a_i(n)$ which gives an isomorphism between rational functional subfields $\mathbf{A}_{a,q}$ and $\C(a) := \C(a_i(n); i \in S, n \in \Z)$. 

Further, we shall transform the (log-canonical) Poisson structure \eqref{Poisson-A} on $\mathbf{A}_{a,q}$ to that on $\C(a)$ by transforming $\delta(q^m)$ into $\delta_{m,0}$ with $\phi_a$, precisely, 
$$
  \{A_{i,aq^{2n+i}}, A_{j,aq^{2n'+j}} \} = 
  \left(\delta(q^{{B}_{ij}+(2n+i)-(2n'+j)}) - \delta(q^{-{B}_{ij}+(2n+i)-(2n'+j)})\right) A_{i,aq^{2n+i}} A_{j,aq^{2n'+j}}
$$   
is transformed into 
\begin{align}\label{eq:a-Poisson}
  \{a_i(n), a_j(n')\} =
  \left(\delta_{{B}_{ij}+(2n+i)-(2n'+j)} - \delta_{-{B}_{ij}+(2n+i)-(2n'+j)} \right) a_i(n) a_j(n').
\end{align}
Thus in the current case of $\mg = A_\ell$, the Poisson structure on $\C(a)$ is obtained as
\begin{align*}
&\{a_i(n), a_i(n') \} = (\delta_{n', n+1} - \delta_{n',n-1}) a_i(n) a_i(n');~ i \in S,
\\  
&\{a_i(n), a_{i+1}(n') \} = (\delta_{n', n-1} - \delta_{n',n}) a_i(n) a_{i+1}(n');~i \in S \setminus \{\ell \}
\end{align*} 
and the others are Poisson commutative.
On the other hand, the exchange matrix of the infinite quiver $Q(\mg)$ is expressed as
\begin{align*}
  &\ve_{v^i_n, v^i_{n'}} = \delta_{n',n+1} - \delta_{n',n-1}; ~ i \in S,
  \\
  &\ve_{v^i_n, v^{i+1}_{n'}} = \delta_{n',n-1} - \delta_{n',n}; ~ i \in S \setminus \{\ell\},  
\end{align*}
and the others are zero.
Thus we obtain the following:

\begin{prop}[Cf. \cite{IH00,I02}]\label{prop:Poisson-aX-A}
Let $\C(\mathbf{X})$ be the rational functional field generated by the $X$-variables for the infinite quiver $Q(A_\ell)$, equipped with the Poisson structure \eqref{Poisson-XA}. 
Define an isomorphism of rational functional fields $\beta : \C(\mathbf{X}) \to \C(a)$ by $X^i_n \mapsto a_i(n)^{-1}$. 
Then $\beta$ is a Poisson map. 
\end{prop}

\begin{remark}
There is another Poisson isomorphism map $\beta': \C(\mathbf{X}) \to \C(a)$ given by $X^i_n \mapsto a_i(n)$, since the log-canonical Poisson bracket is invariant under an inverse transformation of variables;
$\{x, y\} = \ve x y \Leftrightarrow \{\frac{1}{x}, \frac{1}{y}\} = \ve \frac{1}{x} \frac{1}{y} $. 
\end{remark}

\begin{remark}
The map $\phi_a$ \eqref{qY-to-y} already appeared in \cite{I02} for a general $\mg$, from the view point of the lattice $\mathcal{W}$-algebra and the lattice Toda field theory. 
Except for the aspect of cluster algebras, what presented in this subsection and \S \ref{subsec:q-lattice-g} is based on \cite{IH00, I02}.
The variable $\beta^i_n$ there corresponds to $X^i_n$ here.
\end{remark}

\subsection{Weyl group action on $\C(y)$: $\mathfrak{g}=A_\ell$}

We set $q$ to be a root of unity, $\ve^{2m}=1$, with
a positive integer $m$ bigger than $\ell-1$. This setting corresponds to periodicity in $\mathbf{Y}_{a,\ve}$; $Y_{i,a \ve^{n+2m}} = Y_{i,a \ve^n}$. 
Equivalently, we assume periodicity: $y_i(n+m) = y_i(n)$ in $\C(y)$, and $a_i(n+m) = a_i(n)$ in $\C(a)$. 

For $i \in S$, we define rational maps $r_i$ on $\C(y)$ given by
\begin{align}
\label{Weyl-y-A}
r_i(y_j(n)) = 
\begin{cases}
\displaystyle{\frac{f_y(i,n-2)}{f_y(i,n-1)} \frac{y_{i-1}(n) y_{i+1}(n-1)}{y_i(n-1)}} & j = i,
\\[1mm]
y_j(n) & j \neq i,
\end{cases}
\end{align}
where $f_y(i,n) := 1 + \sum_{k=0}^{m-2} (a_i(n) a_i(n-1) \cdots a_i(n-k))^{-1}$.

\begin{thm}
Let $\C(\mathbf{X})$ be the rational functional field generated by the $X$-variables for the quiver $Q_{m}(A_\ell)$.  
The following diagram is commutative:
\begin{align}\label{Weyl-y-X}
\xymatrix{
\C(\mathbf{X}) \ar[r]_{\beta} \ar[d]^{R_i^\ast} & 
\C(a) \ar@{}[r]|{\subset} \ar[d]^{r_i|_{\C(a)}} & \C(y) \ar[d]^{r_i}
\\
\C(\mathbf{X}) \ar[r]_{\beta} &
\C(a) \ar@{}[r]|{\subset} & \C(y)
}
\end{align}
for $i \in S$.
The rational maps $r_i$ \eqref{Weyl-y-A} for $i \in S$ generate the action of $W(A_\ell)$ on $\C(y)$.
\end{thm}

\begin{proof}
To see the commutativity, we calculate $r_i(a_j(n))$:
\begin{align*}
r_i(a_j(n)) 
&= r_i \left( \frac{y_j(n) y_{j}(n+1)}{y_{j-1}(n+1) y_{j+1}(n)} \right)
\\
&= \begin{cases}
    \frac{\frac{f_y(i,n-2)y_{i-1}(n) y_{i+1}(n-1)}{f_y(i,n-1) y_i(n-1)}
          \frac{f_y(i,n-1)y_{i-1}(n+1) y_{i+1}(n)}{f_y(i,n) y_i(n)}}
         {y_{i-1}(n+1) y_{i+1}(n)}
    = \frac{f_y(i,n-2)}{f_y(i,n) a_i(n-1)} & j=i,
    \\[2mm]
    \frac{y_{i+1}(n) y_{i+1}(n+1)}{\frac{f_y(i,n-1)y_{i-1}(n+1) y_{i+1}(n)}{f_y(i,n) y_i(n)} y_{i+2}(n)}
    = \frac{a_{i+1}(n) a_i(n) f_y(i,n)}{f_y(i,n-1)} & j=i+1,
    \\[4mm]
    \frac{y_{i-1}(n) y_{i-1}(n+1)}{ y_{i-2}(n+1) \frac{f_y(i,n-2)y_{i-1}(n) y_{i+1}(n-1)}{f_y(i,n-1) y_i(n-1)}}
    = \frac{a_{i-1}(n) a_i(n-1) f_y(i,n-1)}{f_y(i,n-2)} & j=i-1,
    \\
    a_j(n) & \text{otherwise}.
\end{cases}
\end{align*}
Thus it follows that $r_i(\beta(X^j_n)) = \beta(R_i^\ast(X^j_n))$, using 
$\beta(f_X(i,n)) = f_y(i,n)$. 

Next we check that $r_i r_{i+1} r_i (y_j(n)) = r_{i+1} r_i r_{i+1} (y_j(n))$ holds. Due to Theorem \ref{thm:R-Weyl} it holds that non-trivial identities
$$
  R_i^\ast R_{i+1}^\ast R_{i}^\ast (X_n^{j}) 
  = 
  R_{i+1}^\ast R_{i}^\ast R_{i+1}^\ast (X_n^{j})  
$$
for $j=i, i \pm 1, i+2$, which are equivalent to
\begin{align}
  \label{A-Weyl1}
  &R_{i}^\ast \left(\frac{f_X(i+1,n-3)}{f_X(i+1,n-2)} R_{i+1}^\ast \left(\frac{f_X(i,n-2)}{f_X(i,n-1)}\right) \right)
  =
  \frac{f_X(i+1,n-3)}{f_X(i+1,n-2)} R_{i+1}^\ast \left(\frac{f_X(i,n-2)}{f_X(i,n-1)}\right),
  \\
  \label{A-Weyl2}
  &R_{i+1}^\ast \left(\frac{f_X(i,n-1)}{f_X(i,n)} R_{i}^\ast \left(\frac{f_X(i+1,n-1)}{f_X(i+1,n)}\right) \right)
  =
  \frac{f_X(i,n-1)}{f_X(i,n)} R_{i}^\ast \left(\frac{f_X(i+1,n-1)}{f_X(i+1,n)}\right).
\end{align}
On the other hand, we have 
\begin{align*}
  &r_i r_{i+1} r_i (y_i(n)) 
  = 
  \frac{y_{i-1}(n) y_{i+2}(n-2)}{y_{i+1}(n-2)}
  r_i \left(\frac{f_y(i+1,n-3)}{f_y(i+1,n-2)} r_{i+1} \left(\frac{f_y(i,n-2)}{f_y(i,n-1)}\right) \right),
  \\
  &r_{i+1} r_{i} r_{i+1} (y_i(n)) 
  = 
  \frac{y_{i-1}(n) y_{i+2}(n-2)}{y_{i+1}(n-2)}
  \frac{f_y(i+1,n-3)}{f_y(i+1,n-2)} r_{i+1} \left(\frac{f_y(i,n-2)}{f_y(i,n-1)}\right).
\end{align*}
Then $r_i r_{i+1} r_i (y_i(n)) = r_{i+1} r_{i} r_{i+1} (y_i(n))$ follows from 
\eqref{A-Weyl1} via $\beta$.
In the same manner, $r_i r_{i+1} r_i (y_{i+1}(n)) = r_{i+1} r_{i} r_{i+1} (y_{i+1}(n))$ follows from \eqref{A-Weyl2}. 
For the other $j$ it holds that $r_i r_{i+1} r_i (y_j(n)) = r_{i+1} r_{i} r_{i+1} (y_j(n)) = y_j(n)$, and the claim follows.
\end{proof}

\subsection{Correspondence of the $q$-deformation and the lattice: the case of general $\mg$}\label{subsec:q-lattice-g}

For a generic $q$ and $a \in \C^\times / q^{d\Z}$, define a rational functional field $\mathbf{Y}_{a,q}$:
\begin{align}\label{eq:Y-def-g}
\mathbf{Y}_{a,q} := 
\begin{cases}
\C(Y_{i,a q^{2n+i-1}}; i \in S, n \in \Z) & \mg = A_\ell, C_\ell, G_2,
\\
\C(Y_{i,a q^{2n+i-1}}; i = 1,\ldots \ell_0, Y_{i,a q^{2n+i-2}};i=\ell_0+1,\ldots,\ell; n \in \Z) & \mg = D_\ell, E_{6,7,8},
\\
\C(Y_{i,a q^{2n+i-1}}; i \in S \setminus \{\ell\}, Y_{\ell,a q^{2n+\ell-\frac{3}{2}}}; n \in \frac{\Z}{2}) & \mg = B_\ell,
\\
\C(Y_{i,a q^{2n+i-1}}; i = 1,2,Y_{3,a q^{2n+\frac{3}{2}}},Y_{4,a q^{2n+2}}; n \in \frac{\Z}{2}) & \mg = F_4,
\end{cases}
\end{align}
where $\ell_0=\ell-2$ for $D_\ell$ and , $\ell_0=3$ for $E_{6,7,8}$.
The following lemma is proved case by case (cf. \cite[\S 2.3.2]{HL16}).

\begin{lem}
If $Y_{i,aq^{k}} \in \mathbf{Y}_{a,q}$, then it holds that $A_{i,aq^{k+d_i}} \in \mathbf{Y}_{a,q}$. 
\end{lem}

Define $\mathbf{A}_{a,q} = \C(A_{i,a q^{k+d_i}}; Y_{i,a q^{k}} \in \mathbf{Y}_{a,q}, i \in S, k \in d \Z)$,   
as a subfield of $\mathbf{Y}_{a,q}$. 
We introduce commuting variables $y_i(n)$ on a lattice $(i,n) \in S \times d \Z$, and define rational functional field
$\C(y) := \C(y_i(n); i \in S, n \in d \Z)$. 
Further we define $a_i(n) \in \C(y)$ by
\begin{align}\label{a-y}
  a_i(n) = \frac{y_i(n) y_{i}(n+d_i)}{F(i,n)} 
\end{align}
for $i \in S$ and  $n \in d \Z$, where 
\begin{align}\label{eq:a-y-g}
F(i,n)
=
\begin{cases}
\displaystyle{\prod_{j: v^j_n \leftarrow v^i_n} y_j(n+d_j) \prod_{j: v^j_n \rightarrow v^i_n} y_j(n)} & \text{ almost cases, }
\\
y_{i-1}(n+\frac{1}{2}) y_{i+1}(n) & i=\ell; B_\ell, \,i= 3; F_4,
\\
y_{i-1}(n+1) y_{i+1}(n) y_{i+1}(n+\frac{1}{2}) & i=\ell-1; B_\ell, \,i=2; F_4, 
\\
y_{\ell-1}(n+1) y_{\ell-1}(n+2) & i= \ell; C_\ell,
\\
y_1(n+1)y_1(n+2)y_1(n+3) & i=2; G_2.      
\end{cases}
\end{align}
Let $\phi_a : \mathbf{Y}_{a,q} \to \C(y)$ be an isomorphism of the rational functional fields given by \eqref{qY-to-y} for almost cases, except for 
\begin{align*}
  &Y_{i,a q^{2n+i-2}} \mapsto y_i(n) \quad i=\ell_0+1 \cdots, \ell ~\text{ for }  D_\ell, E_{6,7,8},
\\
  &Y_{\ell,2n+\ell-\frac{3}{2}} \mapsto y_\ell(n) ~~\text{ for } B_\ell,
\\
  &Y_{3,a q^{2n+\frac{3}{2}}} \mapsto y_3(n), ~Y_{4,a q^{2n+2}} \mapsto y_4(n) 
  ~~\text{ for } F_4. 
\end{align*}
By direct calculations, one obtain the following:

\begin{lem}\label{lem:AY-ay}
If $Y_{i,aq^k} \in \mathbf{Y}_{a,q}$ and $\phi_a(Y_{i,aq^k}) = y_i(n)$,
then $\phi_a(A_{i,aq^{k+d_i}}) = a_i(n)$. 
\end{lem}

We transform the Poisson structure \eqref{Poisson-A} on $\mathbf{A}_{a,q}$ to that \eqref{eq:a-Poisson} on $\C(a)$ in the same manner as in the case of $A_\ell$, and obtain the following.

\begin{prop}[\cite{IH00,I02}]
The Poisson structure on $\C(a)$ is obtained as
\begin{align*}
&\{a_i(n), a_i(n') \} = (\delta_{n', n+d_i} - \delta_{n',n-d_i}) a_i(n) a_i(n'),
\\  
&\{a_i(n), a_{j}(n') \} = (\delta_{n', n+B_{ij}} - \delta_{n',n}) a_i(n) a_{j}(n'); ~i < j, ~d_i \leq d_j,
\\
&\{a_i(n), a_{j}(n') \} = (\delta_{n', n+\frac{B_{ij}}{2}} - \delta_{n',n-\frac{B_{ij}}{2}}) a_i(n) a_{j}(n'); ~i < j, ~d_i > d_j,
\end{align*} 
and the others are Poisson commutative.
\end{prop}

Finally, in the same manner as Proposition \ref{prop:Poisson-aX-A}, we prove the following.

\begin{prop}\label{prop:PoissonaX-g}
Let $\C(\mathbf{X})$ be the rational functional field generated by the $X$-variables for the quiver $Q(\mg)$, equipped with the Poisson structure \eqref{Poisson-XA}. An isomorphism of rational functional fields $\beta : \C(\mathbf{X}) \to \C(a)$ given by $X^i_n \mapsto a_i(n)^{-1}$ is a Poisson map. 
\end{prop}

\subsection{Weyl group action on $\C(y)$: the case of general $\mathfrak{g}$}

We consider the case that $q$ is a root of unity, $\ve^{2d'm}=1$, 
where $m$ is a positive integer bigger than $\ell-1$. This corresponds to periodicity in $\mathbf{Y}_{a,\ve}$; $Y_{i,a \ve^{n+2d'm}} = Y_{i,a\ve^n}$. Equivalently, we assume periodicity: $y_i(n+d'm) = y_i(n)$ in $\C(y)$, and also in $\C(a)$.

For $i \in S$, we define rational maps $r_i$ on $\C(y)$ by
\begin{align}
\label{Weyl-y}
r_i(y_j(n)) = 
\begin{cases}
\displaystyle{\frac{f_y(i,n-2d_i)}{f_y(i,n-d_i)} \frac{F(i,n-d_i)}{y_i(n-d_i)}} & j = i,
\\[1mm]
y_j(n) & j \neq i,
\end{cases}
\end{align}
where $F(i,n)$ is defined at \eqref{eq:a-y-g}. We set 
$$
  f_y(i,n) := 1 + \sum_{k=0}^{\frac{d'm}{d_i}-2} (a_i(n) a_i(n-d_i) \cdots a_i(n-d_i k))^{-1}.
$$

\begin{thm}\label{thm:WonY-g}
Let $\C(\mathbf{X})$ be the rational functional field generated by the $X$-variables for the quiver $Q_{m}(\mg)$.  
Then the diagram \eqref{Weyl-y-X} is commutative for $i \in S$.
The rational maps $r_i$ \eqref{Weyl-y} generate the action of $W(\mg)$ on $\C(y)$.
\end{thm}

\begin{proof}
Note that \eqref{Weyl-y} is written as $r_i(y_i(n)) = \frac{f_y(i,n-2d_i)}{f_y(i,n-d_i)} \frac{y_i(n)}{a_i(n-d_i)}$, and that $F(i,n)$ \eqref{eq:a-y-g} does not include any $y_i(k)$. 
For $i \in S$, we have 
\begin{align*}
r_i(a_i(n)) &= \frac{r_i(y_i(n) y_i(n+d_i))}{F(i,n)}
\\
&= \frac{f_y(i,n-2d_i)}{f_y(i,n-d_i) a_i(n-d_i)} \frac{f_y(i,n-d_i)}{f_y(i,n) a_i(n)} a_i(n) 
= \frac{f_y(i,n-2d_i)}{f_y(i,n) a_i(n-d_i)}. 
\end{align*}
For $(i,j)$ such that $i \neq j$ and $d_i \geq d_j$, we have
\begin{align*}
r_i(a_j(n)) = 
\begin{cases}
\displaystyle{\frac{y_j(n) y_j(n+d_i)}{r_i(y_i(n+d_i)) F(i,n)} y_i(n+d_i)
= \frac{f_y(i,n)}{f_y(i,n-d_i)} a_i(n) a_j(n)} & v^j_n \leftarrow v^i_n,
\\[2mm]
\displaystyle{\frac{y_j(n) y_j(n+d_i)}{r_i(y_i(n)) F(i,n)} y_i(n)
= \frac{f_y(i,n-d_i)}{f_y(i,n-2d_i)} a_i(n-d_i) a_j(n)} & v^j_n \rightarrow v^i_n.
\end{cases} 
\end{align*}
Thus it follows that 
\begin{align}\label{comm-beta-R}
\beta(R_i^\ast(X^j_n)) = r_i(\beta(X^j_n))
\end{align}
for $i=j$, and for $i \neq j$ and $d_i \geq d_j$. 
In the same manner, we can show \eqref{comm-beta-R} in the case of $d_i < d_j$,
too.
Then the commutativity of the diagram \eqref{Weyl-y-X} is proved.

To prove the second claim, it is satisfactory to show
\begin{enumerate}
\item[(i)] 
$(r_{\ell} r_{\ell-1})^2 (y_j(n)) = (r_{\ell-1} r_\ell)^2 (y_j(n))$ 
for $j=\ell-1,\ell$ in the case of $C_\ell$, 
\item[(ii)] 
$(r_2 r_1)^3 (y_j(n)) = (r_1 r_2)^3 (y_j(n))$
for $j=1,2$ in the case of $G_2$.
\end{enumerate} 
(i): The non-trivial relations 
\begin{align}\label{C-Weyl0} 
  (R_\ell^\ast R_{\ell-1}^\ast)^2 (X^j_n) = (R_{\ell-1}^\ast R_\ell^\ast)^2 (X^j_n); ~ j=\ell-2,\ell-1,\ell
\end{align}
are equivalent to 
\begin{align}
  \label{C-Weyl1}
  \begin{split}
  &R_{\ell}^\ast \left(\frac{f_X(\ell-1,n-4)}{f_X(\ell-1,n-3)} R_{\ell-1}^\ast \left(\frac{f_X(\ell,n-5)}{f_X(\ell,n-3)} R_\ell^\ast\left(\frac{f_X(\ell-1,n-2)}{f_X(\ell-1,n-1)}\right)\right) \right)  \\
  & \quad \qquad = 
  \frac{f_X(\ell-1,n-4)}{f_X(\ell-1,n-3)} R_{\ell-1}^\ast \left(\frac{f_X(\ell,n-5)}{f_X(\ell,n-3)} R_\ell^\ast\left(\frac{f_X(\ell-1,n-2)}{f_X(\ell-1,n-1)}\right)\right),
  \end{split}
  \\
  \label{C-Weyl2}
  \begin{split}
  &R_{\ell-1}^\ast \left(\frac{f_X(\ell,n-5)}{f_X(\ell,n-3)} R_\ell^\ast \left(\frac{f_X(\ell-1,n-3)}{f_X(\ell-1,n-1)} R_{\ell-1}^\ast\left(\frac{f_X(\ell,n-4)}{f_X(\ell,n-2)}\right)\right)\right)
  \\
  & \quad \qquad = 
  \frac{f_X(\ell,n-5)}{f_X(\ell,n-3)} R_\ell^\ast \left(\frac{f_X(\ell-1,n-3)}{f_X(\ell-1,n-1)} R_{\ell-1}^\ast\left(\frac{f_X(\ell,n-4)}{f_X(\ell,n-2)}\right)\right).
  \end{split}  
\end{align}
In precisely, \eqref{C-Weyl1} is equivalent to \eqref{C-Weyl0} of $j=\ell-2$,
\eqref{C-Weyl1} and \eqref{C-Weyl0} of $j=\ell-1$ gives \eqref{C-Weyl2}, 
and \eqref{C-Weyl0} of $j=\ell$ is obtained from \eqref{C-Weyl1} and \eqref{C-Weyl2}.   
On the other hand, we have 
\begin{align*}
  &(r_{\ell-1} r_\ell)^2 (y_\ell(n)) 
  = 
  \frac{1}{y_\ell(n-3)}
  \frac{f_y(\ell,n-5)}{f_y(\ell,n-3)} r_{\ell} \left(\frac{f_y({\ell-1},n-3)}{f_y({\ell-1},n-1)} r_{{\ell-1}} \left(\frac{f_y(\ell,n-4)}{f_y(\ell,n-2)} \right) \right),
  \\
  &(r_\ell r_{\ell-1})^2 (y_\ell(n))  
  = 
  \frac{1}{y_\ell(n-3)}
  r_{\ell-1} \left(\frac{f_y(\ell,n-5)}{f_y(\ell,n-3)} r_{\ell} \left(\frac{f_y({\ell-1},n-3)}{f_y({\ell-1},n-1)} r_{{\ell-1}} \left(\frac{f_y(\ell,n-4)}{f_y(\ell,n-2)} \right) \right) \right).
\end{align*}
Then $(r_\ell r_{\ell-1})^2 (y_\ell(n)) = (r_{\ell-1} r_\ell)^2 (y_\ell(n))$ follows from \eqref{C-Weyl2}. In the same way, $(r_\ell r_{\ell-1})^2 (y_{\ell-1}(n)) = (r_{\ell-1} r_\ell)^2 (y_{\ell-1}(n))$ follows from \eqref{C-Weyl1}.
\\
(ii): We shall consider the quiver $Q' := Q_m(G_2^{(1) \vee})$, obtained from $Q_m(G_2)$ by adding a circle $P_0: v^0_1 \to v^0_2 \to \cdots v^0_{3m} \to v^0_1$ 
and arrows $v^1_n \to v^0_n, ~ v^0_n \to v^1_{n-1}$ for $n \in \Z$.
As remarked in \S \ref{subsec:affine}, the Weyl group action $W(G_2^{(1) \vee}) \subset \Gamma_{Q'}$ is obtained. 
The non-trivial relations on $\X$ for $Q'$ generated by $W(G_2) \subset W(G_2^{(1) \vee})$ are given by  
\begin{align}\label{G-Weyl0} 
  (R_1^\ast R_2^\ast)^3 (X^j_n) = (R_2^\ast R_2^\ast)^3 (X^j_n); ~ j=0,1,2,
\end{align}
and these are equivalent to the following relations among the functions $f^i_n:=f_X(i,n)$ for $i=1,2$: 
\begin{align}
  \begin{split}\label{G-Weyl1}
  &R_2^\ast \left(\frac{f^1_{n-7}}{f^1_{n-6}} R_1^\ast \left( \frac{f^2_{n-9}}{f^2_{n-6}} R_2^\ast \left(\frac{f^1_{n-5}}{f^1_{n-3}} R_1^\ast \left( \frac{f^2_{n-7}}{f^2_{n-4}} R_2^\ast \left(\frac{f^1_{n-2}}{f^1_{n-1}} \right)\right)\right)\right)\right)
  \\
  & \quad = \frac{f^1_{n-7}}{f^1_{n-6}} R_1^\ast \left( \frac{f^2_{n-9}}{f^2_{n-6}} R_2^\ast \left(\frac{f^1_{n-5}}{f^1_{n-3}} R_1^\ast \left( \frac{f^2_{n-7}}{f^2_{n-4}} R_2^\ast \left(\frac{f^1_{n-2}}{f^1_{n-1}} \right)\right)\right)\right),
  \end{split}
  \\
  \begin{split}\label{G-Weyl2}
  &\frac{f^2_{n-9}}{f^2_{n-6}} R_2^\ast \left(\frac{f^1_{n-5} f^1_{n-6}}{f^1_{n-3} f^1_{n-7}} R_1^\ast \left( \frac{f^2_{n-6} f^2_{n-7}}{f^2_{n-4} f^2_{n-9}} R_2^\ast \left(\frac{f^1_{n-2} f^1_{n-3}}{f^1_{n-1}f^1_{n-5}} R_1^\ast \left( \frac{f^2_{n-4}}{f^2_{n-7}} R_2^\ast \left(\frac{f^1_{n}}{f^1_{n-2}} \right)\right)\right)\right)\right)
  \\
  & \quad =
\frac{f^1_{n-5}}{f^1_{n-7}}
R_1^\ast \left(\frac{f^2_{n-5}}{f^2_{n-8}} R_2^\ast \left(\frac{f^1_{n-2} f^1_{n-6}}{f^1_{n-4} f^1_{n-5}} R_1^\ast \left( \frac{f^2_{n-3} f^2_{n-8}}{f^2_{n-5} f^2_{n-6}} R_2^\ast \left(\frac{f^1_{n} f^1_{n-4}}{f^1_{n-1}f^1_{n-2}} R_1^\ast \left( \frac{f^2_{n-6}}{f^2_{n-3}}  \right)\right)\right)\right)\right).
  \end{split}
\end{align}
Actually, \eqref{G-Weyl1} and \eqref{G-Weyl2} are respectively equivalent to \eqref{G-Weyl0} of $j=0$ and $j=1$.  
By combining \eqref{G-Weyl1} and \eqref{G-Weyl2} we obtain \eqref{G-Weyl0} of $j=2$. Note that \eqref{G-Weyl1} and \eqref{G-Weyl2} include no $X^0_n$. 
Further, one sees that $(r_1 r_2)^3 (y_1(n)) = (r_2 r_1)^3 (y_1(n))$ follows from \eqref{G-Weyl1}, and $(r_1 r_2)^3 (y_2(n)) = (r_2 r_1)^3 (y_2(n))$ follows from \eqref{G-Weyl1} and \eqref{G-Weyl2}.
\end{proof}

\subsection{Invariance of $\mathrm{Im} \chi_q$}

In the following we identify $a_i(n)$ with $(X_n^i)^{-1}$ and $f_y(i,n)$ with $f_X(i,n)$, in $\C(y)$ via the map $\beta$.
For a generic $q$ corresponding to the infinite quiver $Q(\mg)$,
from Lemma \ref{lem:AY-ay} it follows that for $Y_{i,a q^k} \in \mathbf{Y}_{a,q}$ we have  
$\phi_a(Y_{i,a q^{k}}(1+A_{i,a q^{k+d_i}}^{-1})) = y_i(n)(1+X_n^i)$ for some $n \in d \Z$. Thus we have $\phi_a(\mathrm{Im} \chi_q \cap \mathbf{Y}_{a,q})$ equal to
$$
  \bY_{\chi_q} := \bigcap_{i \in S} \Z[y_j(n)^\pm; j \neq i, n \in d \Z] \otimes \Z[y_i(n)(1 + X_n^i); n \in d \Z].
$$

Next we consider the case that $q$ is a root of unity, $\ve^{2md'} = 1$, which corresponds to the periodic quiver $Q_m(\mg)$. Due to Remark \ref{rem:q-root}, 
we have $\phi_a(\mathrm{Im} \chi_\ve \cap \mathbf{Y}_{a,\ve})$ equal to
$$
  \bY_{\chi_\ve} := \bigcap_{i \in S} \Z[y_j(n)^\pm; j \neq i, n \in d \Z/d'm\Z] \otimes \Z[y_i(n)(1 + X_n^i); n \in d \Z/d'm\Z].
$$

\begin{prop}\label{prop:qch-weyl}
$\bY_{\chi_\ve}$ is invariant under the action of $W(\mg)$.
\end{prop}

\begin{proof}
It is satisfactory to check that $y_i(n) (1 + X_n^i)$ is an invariance of $r_i$. 
By using an identity: 
\begin{align}\label{Xf-relation}
  f_X(i,n+d_i) + X^i_n f_X(i,n-d_i) = (1 + X^i_{n+d_i}) f_X(i,n),
\end{align}
which can be checked easily, we obtain 
\begin{align*}
r_i \left(y_i(n) (1 + X_n^i)\right) 
&= \frac{f_X(i,n-2 d_i)}{f_X(i,n-d_i)} X_{n-d_i} y_i(n) 
\left( 1 + \frac{f_X(i,n)}{X^i_{n-d_i} f_X(i,n-2 d_i)} \right)
\\
&= \frac{y_i(n)}{f_X(i,n-d_i)}(X^i_{n-d_i} f_X(i,n-2d_i) + f(n))
\\
&\stackrel{\eqref{Xf-relation}}{=} 
y_i(n) (1 + X_n^i).
\end{align*}
\end{proof}

Now we come back to the case of generic $q$, and consider an analogue of the map $r_i$. Define $\C[y^\pm] :=  \C[y_i(n), y_i(n)^{-1}; i \in S, n \in d\Z]$, and let $I_X$ be the ideal of $\C[y^\pm]$ generated by $X^i_n$ for $i \in S, ~n \in d\Z$.
We write $\widehat{\C}_X(y)$ for the quotient field of the completion of $\C[y^{\pm}]$ given by $\lim_{\leftarrow} (\C[y^\pm]/I_X^k)$.
For $(i,n) \in S \times d\Z$, we define an `infinite' version of $f_X(i,n)$ as a formal power series in the $X^i_n$:
$$
  \hat{f}_X(i,n) = 1 + \sum_{k=0}^\infty X^i_n X^i_{n-d_i} \cdots X^i_{n-kd_i}.
$$
For $i \in S$ define a map  $\hat{r}_i : \C(y) \to \widehat{\C}_X(y)$ similarly as \eqref{Weyl-y} by 
\begin{align}\label{eq:r-hat}
\hat{r}_i(y_j(n)) = 
\begin{cases}
\displaystyle{y_i(n) X^i_{n-d_i} \frac{\hat{f}_X(i,n-2d_i)}{\hat{f}_X(i,n-d_i)} }& j = i,
\\[1mm]
y_j(n) & j \neq i.
\end{cases}
\end{align}

\begin{prop}\label{prop:qch-weyl-inf}
$\bY_{\chi_q}$ is invariant under the action of $\hat{r}_i ~(i \in S)$.
\end{prop}

\begin{proof}
It is easy to see that $\hat{f}_X(i,n)$ satisfies the similar relation as \eqref{Xf-relation}:
\begin{align}
  \hat{f}_X(i,n+d_i) + X^i_n \hat{f}_X(i,n-d_i) = (1 + X^i_{n+d_i}) \hat{f}_X(i,n).
\end{align}
The rest is same as the proof of Proposition \ref{prop:qch-weyl}.
\end{proof}

\begin{remark}
The maps $\hat{r}_i$ are not related to mutation sequences, because we have no circle on the infinite quiver $Q(\mg)$ anymore.
Moreover they are not extended to the endomorphism on $\widehat{\C}_X(y)$, 
thus we cannot naively discuss if $\hat{r}_i ~(i \in S)$ are related to the Weyl group.\footnote{~The author thanks Ryo Fujita for pointing out this fact.}
We wonder if there is any correspondence between $\hat{r}_i$ and the screening operator on the lattice \eqref{lattice-screeing}.
\end{remark}

\begin{remark}\label{rem:braidacion}
The Weyl group action on the tropical $X$-variables is related to the braid group action on the $\ell$-integral weight lattice introduced in \cite{CM05} as follows.

First we extend the action $R_i^{\mathrm{trop}}$ on the tropical $X$-seed $(Q_m(\mg),\mathbf{x})$ of positive tropical sign (Corollary \ref{cor:tropR}) to the case of the infinite quiver $Q(\mg)$ as follows: let $I$ be the vertex set of $Q(\mg)$, and set a tropical semifield $\mathbb{P} := \mathbb{P}_{\mathrm{\trop}}(\mathbf{u}); \mathbf{u}=(u^i_n; v^i_n \in I)$. 
Consider $\X_{Q(\mg)}(\mathbb{P})$ and its subset $\mathcal{X}^{i +}_{Q(\mg)}(\mathbb{P}) := \{\mathbf{x} \in \mathbb{P}^{|I|}; ~x^i_n > 0 \text{ for all } n \in d \Z\}$.
The map $\hat{R}_i^{\mathrm{trop}}: \X^{i +}_{Q(\mg)}(\mathbb{P}) \to \X_{Q(\mg)}(\mathbb{P})$ given by \eqref{eq:trop-R1} and \eqref{eq:trop-R2} is well-defined. 
Although we do not know if $\hat{R}_i^{\mathrm{trop}}$ belongs to the cluster modular group of $Q(\mg)$, we can show that $\hat{R}_i^{\mathrm{trop}} ~(i \in S)$ satisfy the braid relations in the same manner as Theorem \ref{thm:R-Weyl}. 

Then, one sees that the map $\hat{R}_i^{\mathrm{trop}}$ is related to the action of $T_i$ on the $\ell$-integral weight lattice \cite[(3.7)]{CM05}, by `identifying' $(x^i_n)^{-1}$ with the $\ell$-simple root $\boldsymbol{\alpha}_{i,a}$ for some $a \in \C^\times$. (This identification is essentially same as what pointed out at \cite[Remark 3.4]{CM05}.)
This is also observed with \eqref{eq:r-hat} formally, by `setting $\hat{f}_X(i,n)$ to be $1$' and `identifying' $y_i(n)$ with the $\ell$-fundamental weight $\boldsymbol{\omega}_{i,a}$ for some $a \in \C^\times$.       
We remark that \cite[Proposition 3.6]{CM05} seems to correspond to Theorem \ref{thm:R-Weyl} (2). 
\end{remark}

\section{Relation with the lattice Toda field theory} 

\subsection{The lattice Toda field}

We shall recall the lattice Toda field theory introduced in \cite{IH00, I02},
by combining with the results obtained in the previous sections.
For a finite dimensional simple Lie algebra $\mg$, 
let $s_i(n)$ be commuting variables on the lattice $(i,n) \in S \times d\Z$, satisfying the following Poisson relations: 
\begin{align}\label{g-x-Poisson}
  \begin{split}
  &\{ s_i(n), s_i(m)\} =  s_i(n) s_i(m); ~n \equiv m(\mathrm{mod } \, d_i),
  \\
  &\{s_i(n), s_j(m)\} = - \frac{1}{2}  s_i(n) s_j(m); ~B_{ij} \neq 0, ~n \equiv m \,(\mathrm{mod ~ min}( d_i, d_j)),
  \\
  &\{s_i(n), s_j(n)\} = - \frac{1}{2} s_i(n) s_j(n); ~i < j, ~B_{ij} \neq 0,
  \end{split}
\end{align}
where we assume $n < m$. The others are Poisson commutative.  
Let $\C(\mathbf{X})$ be the rational functional field generated by the $X$-variables for $Q(\mg)$. 
Define $\C(s):= \C(s_i(n);~i \in S, ~n \in d\Z)$ and an embedding map $\sigma : \C(\mathbf{X}) \to \C(s)$ given by    
\begin{align}\label{X-s}
  X^i_n = \frac{s_i(n)}{s_i(n+d_i)}.
\end{align}

\begin{prop}[Cf. \cite{IH00}]
The map $\sigma$ is a Poisson map.
\end{prop}

Let $\mathcal{H}$ be the Hamiltonian function for the lattice Toda field equation given by
\begin{align}
 \label{g-Hamiltonian}
  \mathcal{H} = \sum_{i \in S} ~ \mathcal{H}_i;
  \qquad
  \mathcal{H}_i = \sum_{n \in d {\mathbb{Z}}} ~ s_i(n),
\end{align}
and define the lattice Toda field equation by 
$\frac{\partial}{\partial t} \log X^i_n := \{ \mathcal{H} , \log X^i_n \}$,
which is a set of differential-difference equations:
\begin{align}
  \label{eq-Toda-classical}
  \begin{split}
  \frac{\partial}{\partial t} \log X^i_n 
  =
  \Bigl(
  \sum_{j: j < i} ~
  \sum_{k=1}^{-C_{ji}}
  ~ s_j(n + k d_{ij})
  \Bigr)
  -
  s_i(n) - s_i(n + d_i)
  +
  \Bigl(
  \sum_{j:j > i} ~
  \sum_{k=0}^{-C_{ji}-1}
  ~ s_j(n + k d_{ij})
  \Bigr),
  \end{split}
\end{align}
where we set $d_{ij} = \min(d_i, d_j)$.

\begin{remark}
By taking a continuum limit of a coordinate $n \in d \Z$ with 
$s_i(n) \to s_i := s_i(z,t)$, $\log X^i_n$ reduces to $-d_i \frac{\partial}{\partial z} \log s_i$, and \eqref{eq-Toda-classical} reduces to the Toda field equation:
\begin{equation}
  \label{g-Toda-eq}
  \frac{\partial^2}{\partial t \partial z} \log \, s_i
   =
   \sum_{k \in S} ~ \frac{C_{ji}}{d_i} s_j.
\end{equation}

\end{remark}

On the other hand, for $i \in S$ define linear operators $\mathcal{S}_i$ as a lattice analogue of \eqref{eq:q-screening}:
\begin{align}\label{lattice-screeing}
  \mathcal{S}_i : \C(y) \to 
                  \C(y)_i := \oplus_{n \in d \Z} \C(y) \otimes s_i(n);
               ~y_j(n) \mapsto \delta_{i,j} y_i(n) s_i(n)
\end{align}
with the Leibniz rule.
Especially we have $\mathcal{S}_i : y_j(n)^{-1} \mapsto -\delta_{i,j} y_j(n)^{-1}s_i(n)$. 
Let $\C(y)_i'$ be the quotient of $\C(y)_i$ with relations
\begin{align}
  s_i(n+d_i) = a_i(n) s_i(n).
\end{align} 

By using \eqref{a-y} and the Poisson isomorphism $\beta : \C(\mathbf{X}) \to \C(a)$, we obtain the following.

\begin{prop}[\cite{I02}]
It holds that $\mathcal{S}_i \cdot X^j_n = \{\mathcal{H}_i, X^j_n\}$ for $i,j \in S$ and $n \in d \Z$.
\end{prop}

\subsection{Cluster realization of the $\tau$-function}

Let $\tau_i(n)$ for $(i,n) \in S \times d\Z$ be a function in $t \in \R$ satisfying the bilinear equation:
\begin{align}\label{gToda-tau-eq}
  D_t \tau_i(n) \cdot \tau_i(n+d_i) 
  =
  \prod_{j:j<i} \prod_{k=1}^{-C_{ij}} \tau_j(n+k d_{ij}) \cdot
  \prod_{j:j>i} \prod_{k=0}^{-C_{ij}-1} \tau_j(n+k d_{ij}),  
\end{align} 
where $D_t$ is Hirota's bilinear operator acting on an ordered pair of two functions $f$ and $g$ as $D_t f \cdot g := \frac{\partial f}{\partial t} g - f \frac{\partial g}{\partial t}$.
We write $T_i(n)$ for the r.h.s. of \eqref{gToda-tau-eq}. 
It is easy to see the following lemma.

\begin{lem}[\cite{IH00}]
Set 
\begin{align}\label{s-tau}
  s_i(n) = \frac{T_i(n)}{\tau_i(n) \tau_i(n+d_i)}.
\end{align}
Then the $X^i_n$ \eqref{X-s} satisfy the lattice Toda field equation \eqref{eq-Toda-classical}.
\end{lem} 
\begin{proof}
For the self-consistency we present the proof briefly.
The bilinear equation \eqref{gToda-tau-eq} is rewritten as
$$
  \frac{\partial}{\partial t} \log\frac{\tau_i(n)}{\tau_i(n+d_i)} = s_i(n).
$$
From this, it follows that
$$
  \frac{\partial}{\partial t} \log \prod_{k=1}^{-C_{ij}}
  \frac{\tau_i(n+k d_{ij})}{\tau_i(n+k d_{ij}+d_i)}
  = \sum_{k=1}^{-C_{ji}} s_j(n+kd_{ij}),
$$
and we obtain the claim by calculating $\frac{\partial}{\partial t} X^i_n$ using \eqref{X-s} and \eqref{s-tau}.  
\end{proof}

We call $\tau_i(n)$ the {\em $\tau$-function} for the lattice Toda field equation \eqref{eq-Toda-classical}.
We are going to relate the $\tau$-function $\tau_i(n)$ with the $A$-variables for the infinite quiver $Q(\mg)$.
Let $\C(\tau)$ be the rational functional field generated by $\tau_i(n)$ for 
$(i,n) \in S \times d \Z$.  
For a seed $(Q(\mg),\mathbf{X},\mathbf{A})$, recall the positive map $p$ \eqref{eq:positivemap}.

\begin{prop}\label{prop:Toda-A}
Let $\phi$ be a map from $\C(\tau)$ to $\C(\mathbf{A})$ given by
$\tau_i(n) \mapsto A^i_{n-d_i}$. Define a map $p_\tau : \C(\mathbf{X}) \to \C(\tau)$ given as a composition of \eqref{X-s} and \eqref{s-tau}.
Then the following diagram is commutative:
\begin{align*}
\xymatrix{
\C(\mathbf{X}) \ar[rd]_{p^\ast} \ar[r]^{p_\tau} & 
\C(\tau) \ar[d]^{\phi} 
\\
&  \C(\mathbf{A}) 
\\
}
\end{align*}
\end{prop}

\begin{proof}
We have $p_\tau(X^i_n)$ expressed as
\begin{align}
p_\tau(X^i_n) =
  \frac{\tau_i(n+2d_i)}{\tau_i(n)} \cdot   
  \begin{cases}
  \displaystyle{\prod_{j:j \lhd i} \frac{\tau_j(n+d_{j})}{\tau_j(n+d_{j}+d_i)}  
  \prod_{j:j \rhd i} \frac{\tau_j(n)}{\tau_j(n+d_i)}} & \text{almost cases},
  \\
  \displaystyle{\frac{\tau_{\ell-1}(n+\frac{1}{2})}{\tau_{\ell-1}(n+\frac{3}{2})}} & i =\ell \text{ for } B_\ell,
  \\
  \displaystyle{\frac{\tau_{\ell-2}(n+1) \tau_{\ell}(n)}{\tau_{\ell-2}(n+2) \tau_{\ell}(n+2)}} & i =\ell-1 \text{ for } C_\ell,
  \\
  \displaystyle{\frac{\tau_{2}f(n+\frac{1}{2}) \tau_4(n)}{\tau_{\ell-1}(n+\frac{3}{2}) \tau_{4}(n+\frac{1}{2})}} & i =3 \text{ for } F_4,
  \\
  \displaystyle{\frac{\tau_{2}(n)}{\tau_{2}(n+3)}} & i =1 \text{ for } G_2.
  \end{cases}
\end{align}
Here $j \lhd i$ is defined at Theorem \ref{R-action}.
It is easy to check $\phi \circ p_\tau(X^i_n) = p^\ast(X^i_n)$ case by case.
For example, in the case of $\mg$ of a simply-laced Dynkin diagram, we obtain
$$
  \phi \circ p_\tau(X^i_n) 
  = 
  \frac{A^i_{n+1}}{A^i_{n-1}} 
  \prod_{j:j \lhd i} \frac{A^j_n}{A^j_{n+1}}  
  \prod_{j:j \rhd i} \frac{A^j_{n-1}}{A^j_n}, 
$$ 
which coincides with $p^\ast(X^i_n)$. 
\end{proof}


\end{document}